\documentclass[a4paper,10pt]{amsart}

\usepackage[T1]{fontenc}

\usepackage[utf8]{inputenc}

\usepackage{lmodern}

\usepackage[english]{babel}

\usepackage{amsmath}

\usepackage{amssymb}

\usepackage{mathtools}

\usepackage{mathrsfs}

\usepackage{bookmark}

\usepackage{hyperref} 
\hypersetup{
	colorlinks=true,
	linkcolor=blue,
	filecolor=magenta,      
	urlcolor=cyan,
}

\usepackage[alphabetic]{amsrefs}

\usepackage{enumitem}
\usepackage{graphicx}
\usepackage{tikz}
\usepackage{tikz-cd}
\usetikzlibrary{matrix,arrows,decorations.pathmorphing}
\usepackage[all,cmtip]{xy}

\newtheorem*{thm-no-num}{Theorem}
\newtheorem*{df-no-num}{Definition}
\newtheorem{thm}{Theorem} [section]
\newtheorem{prop}[thm]{Proposition} 
\newtheorem{lm}[thm]{Lemma} 
\newtheorem{cor}[thm]{Corollary} 
\theoremstyle{remark}
\newtheorem{rmk}[thm]{Remark}

\theoremstyle{definition}

\newcommand{\bA}{\mathbb{A}}
\newcommand{\PP}{\mathbb{P}}
\newcommand{\ZZ}{\mathbb{Z}}

\newcommand{\OO}{\mathcal{O}}
\newcommand{\cl}[1]{\mathcal{#1}}
\newcommand*{\sheafhom}{\mathcal{H}\kern -.5pt om}
\newcommand{\Brm}{{\rm B}}

\renewcommand{\Brm}{{\rm B}}
\newcommand{\Ccal}{\cl{C}}

\newcommand{\Hcal}{\cl{H}}

\newcommand{\Xcal}{\cl{X}}
\newcommand{\Mcal}{\cl{M}}

\newcommand{\Gm}{\mathbb{G}_m}
\newcommand{\GLt}{\textnormal{GL}_3}
\newcommand{\PGLt}{\textnormal{PGL}_2}

\newcommand{\sm}{_{\rm sm}}

\newcommand{\Het}{{\rm H}^{\bullet}}
\newcommand{\Inv}{{\rm Inv}^{\bullet}}
\newcommand{\Spec}{{\rm Spec}}

\newcommand{\del}{\partial}
\begin{document}
	\title[A complete description of $\Inv(\Hcal_g)$]{A complete description of the cohomological invariants of even genus hyperelliptic curves}
	\author[A. Di Lorenzo]{Andrea Di Lorenzo}
	\address{Aarhus University, Ny Munkegade 118, DK-8000 Aarhus C, Denmark}
	\email{andrea.dilorenzo@math.au.dk}
	\author[R. Pirisi]{Roberto Pirisi}
	\address{KTH Royal Institute of Technology, Lindstedtsvägen 25, 10044 Stockholm, Sweden}
	\email{pirisi@kth.se}	
	\date{\today}
\begin{abstract}
When the genus $g$ is even, we extend the computation of mod $2$ cohomological invariants of $\Hcal_g$ to non algebraically closed fields, we give an explicit functorial description of the invariants and we completely describe their multiplicative structure. 

In the Appendix, we show that the cohomological invariants of the compactification $\overline{\Hcal}_g$ are trivial, and use our methods to give a very short proof of a result by Cornalba on the Picard group of the compactification $\overline{\Hcal}_g$ and extend it to positive characteristic.
\end{abstract}
\maketitle

\section*{Introduction}

Cohomological invariants of algebraic groups are a well-known arithmetic analogue to the theory of characteristic classes for topological groups. The category of topological spaces is replaced with extensions of a base field $k$, and singular cohomology is replaced with Galois cohomology. More precisely, given an algebraic group $G$, write $P_{\Brm G}$ for the functor that associates to a field $K/k$ the set of isomorphism classes of $G$-torsors over $K$, and given a prime number $\ell$ not divisible by the characteristic of our base field, consider the functor 
\[
{\rm H}_{\ell}: \left( \textnormal{Field}/k\right) \rightarrow \left( \textnormal{Set} \right)
\]
\[
K \mapsto\oplus_i\Het_{\ell}(K) \overset{\rm def}{=} \oplus_i {\rm H}_{\textnormal{Gal}}^i(K,\ZZ/\ell\ZZ(i))
\]
where $\ZZ/\ell\ZZ(i)$ denotes the Tate twist $\ZZ/\ell\ZZ \otimes \mu_{\ell}^{\otimes i}$.

\begin{df-no-num}\label{def:coh inv BG}
	A cohomological invariant of $G$ with coefficients in ${\rm H}_{\ell}$ is a natural transformation \[P_{\Brm G} \to {\rm H}_{\ell}\] of functors from fields over $k$ to sets.
\end{df-no-num} 
For brevity's sake, we will often refer to these objects as cohomological invariants modulo $\ell$. The set of cohomological invariants has a natural structure of graded-commutative ring induced by the structure of $\Het_{\ell}(-)$.

The first appearance of cohomological invariants can be traced back to the seminal paper \cite{Wit} and since then they have been extensively studied. The book \cite{GMS} by Garibaldi, Merkurjev and Serre provides a detailed introduction to the modern approach to this theory.

One can think of the cohomological invariants of $G$ as invariants of the classifying stack $\Brm G$ rather than the group $G$. Following this idea, in \cite{PirAlgStack} the second author extended the notion of cohomological invariants to arbitrary smooth algebraic stacks over $k$:
\begin{df-no-num}
	Let $\Xcal$ be an algebraic stack over $k$, and let $P_{\Xcal}: \left( \mathrm{Field}/k \right) \to \left( \mathrm{Set} \right)$ be its functor of points. A cohomological invariant with coefficients in ${\rm H}_{\ell}$ of $\Xcal$ is a natural transformation 
	\[ P_{\Xcal}\longrightarrow {\rm H}_{\ell} \]
	satisfying a certain continuity condition (see \cite{PirAlgStack}*{Def. 2.2}).
\end{df-no-num} 
Note that this definition recovers the classical invariants by taking $\Xcal=\Brm G$. The graded-commutative ring of cohomological invariants of a smooth algebraic stack $\Xcal$ is denoted $\Inv(\Xcal, {\rm H}_{\ell})$. In particular, the ring of classical cohomological invariants mod $\ell$ of a group $G$ is here denoted $\Inv(\Brm G, {\rm H}_{\ell})$ rather than $\Inv(G, \ZZ/\ell\ZZ)$, which is the notation adopted in \cite{GMS}.

By \cite[4.9]{PirAlgStack} the cohomological invariants of a smooth scheme $X$ are equal to its zero-codimensional Chow group with coefficients \[A^0(X,{\rm H}_{\ell}),\] a generalization of ordinary Chow groups introduced by Rost \cite{Rost}. Given a smooth quotient stack $\Xcal = \left[ X/G \right]$ we can construct the equivariant Chow ring with coefficients $A^{*}_G(X,{\rm H}_{\ell})$ following Edidin and Graham's construction \cite{EG} and we have the equality $A^{0}_G(X,{\rm H}_{\ell})=\Inv(\Xcal,{\rm H}_{\ell})$ by \cite[2.10]{PirCohHypEven}. 

Chow groups with coefficients share most of the properties of ordinary Chow groups, and in fact have some stronger ones, such as having a long exact localization sequence rather than a short one.

In \cite{PirAlgStack} the second author also computed the cohomological invariants of $\cl{M}_{1,1}$, the moduli stack of smooth elliptic curves. In the subsequent works \cite{PirCohHypEven} and \cite{PirCohHypThree} he computed the cohomological invariants of $\Hcal_g$, the moduli stack of smooth hyperelliptic curves, when $g$ is even or equal to $3$ and the base field is algebraically closed. The first author then extended the result to arbitrary odd genus \cite{DilCohHypOdd}, using a new presentation of the stack $\Hcal_g$ he developed in \cite{DilChowHyp}. When $\ell$ is odd, the invariants turn out to be (almost) trivial, and moreover the computations work for arbitrary fields. When $\ell=2$ we have nontrivial invariants of degree up to $g+2$. Some relevant questions are still open in the case $\ell=2$:

\begin{itemize}
\item Does the result work for non algebraically closed fields?
\item Can we get an explicit description of the invariants?
\item What is the multiplicative structure of $\Inv(\Hcal_g,{\rm H}_{2})$?
\end{itemize}

This paper answers the three questions above when $g$ is even (see Section \ref{sec:mult}). The main idea is rather simple: given a hyperelliptic curve $C$ over a field $K$, consider the curve's Weierstrass divisor $W_C$, i.e. the ramification divisor of the quotient map $C \rightarrow C/\iota$ given by the hyperelliptic involution. Then $W_C$ is the spectrum of an \'etale algebra of degree $2g+2$ over $K$, which corresponds to a ${\rm S}_{2g+2}$-torsor. 

The resulting map $\Hcal_g \rightarrow \Brm {\rm S}_{2g+2}$ produces an inclusion $\Inv(\Brm {\rm S}_{2g+2}) \subset \Inv(\Hcal_g)$ which yields $\Het_2(k)$-linearly independent invariants $\alpha_0=1, \alpha_1, \ldots, \alpha_{g+1}$, respectively of degree $0, \ldots, g+1$ (see Section \ref{sec:InvWei}). 

These invariants turn out to almost generate $\Inv(\Hcal_g)$: there is only one missing generator, of degree $g+2$, of which we give an explicit description.

Specifically, we can do the following. Assume that $g$ is even. A hyperelliptic curve over $K$ comes equipped with a rational conic $C'=C/\iota$ over $K$, an invertible sheaf $L$ of degree $-g-1$ on $C'$, and a section $s$ of ${\rm H}^0(C',L^{\otimes -2})$. We can (smooth-Nisnevich) locally on $\Hcal_g$ choose a section $s_0$ of $L^{\otimes -2}$. Then the element $t(C) := s/s_0$ can be seen as belonging to ${\rm H}^1(K)=K^*/(K^*)^2$. The product $t \cdot \alpha_{g+1}$ does not depend on the choices we made and provides a new invariant $\beta_{g+2}$.

Another way of seeing the same invariant is that locally we can assume that our section does not pass through a given point $\infty$ of $C'$. Then $s(\infty)$ is well defined up to squares and the product $s(\infty) \cdot \alpha_{g+1}$ can be extended to our last invariant $\beta_{g+2}$. 

This approach works over any field, solving the first two questions. For the last one, the multiplicative structure of $\Inv(\Brm {\rm S}_{2g+2})$ is known, and their products with $\beta_{g+2}$ can be easily obtained from the explicit description we sketched above, completely determining in this way the multiplicative structure of $\Inv(\Hcal_g)$ when $g$ is even (see Theorem \ref{thm:mult}). 

When $g$ is odd we still have an injective map $\Inv(\Brm {\rm S}_{2g+2}) \rightarrow \Inv(\Hcal_g)$, and we can use the same techniques to compute the cohomological invariants up to at most a last generator of degree $g+2$. Unfortunately the approach used in this paper to produce the last generator fails for $g$ odd. Nevertheless, we can still say something about the $g$ odd case assuming that the base field is algebraically closed, and we also outline a possible approach to the problem for constructing the last invariant over general fields.

In Appendix \ref{app:InvHbar}, we show that the cohomological invariants of the compactification $\overline{\Hcal}_g$ are trivial. In Appendix \ref{app:PicHbar}, we use a similar equivariant argument to produce a short proof of a result of Cornalba that the Picard group of $\overline{\Hcal}_g$, the stack of stable hyperelliptic curves, is torsion free over a field of characteristic zero. Moreover we extend the result to any field of characteristic different from two.

\subsection*{Acknowledgements}
The main idea of this paper came to the second author thanks to a discussion with Emiliano Ambrosi during the conference "Cohomology of algebraic varieties" at CIRM, Marseille. We are grateful to both him and the organizers of the conference. We thank Angelo Vistoli for useful comments on Appendix \ref{app:PicHbar}. We also thank the anonymous referees, whose comments much improved the quality of the exposition.

\subsection*{Notation} We fix a base field $k$ of characteristic different from $2$. Every scheme and stack is assumed to be of finite type over $\Spec(k)$. By $\ell$ we will always denote a prime number not divisible by the characteristic of $k$.

Given a scheme $X$, with the notation $\Het_{\ell}(X)$ we will always mean the graded-commutative ring $\oplus_i {\rm H}^i_{\'{e}t}(X,\ZZ/\ell\ZZ(i))$. Sometimes, we will write $\Het_{\ell}(R)$ , where $R$ is a finitely generated $k$-algebra, to denote $\Het_{\ell}(\Spec(R))$.

When $\ell=2$, we will shorten $\Inv(X, {\rm H}_{2})$ to $\Inv(X)$ and $A^{*}(X,{\rm H}_{2})$ to $A^{*}(X)$.

\section{Preliminaries}

\subsection{The moduli stacks of hyperelliptic curves}\label{sec:pres Hg}

We begin by recalling the presentation of $\Hcal_g$ by Arsie and Vistoli \cite{ArsVis}. Let $n$ be an even positive integer, and consider the affine space $\bA^{n+1}$, seen as the space of binary forms of degree $n$. There are two different natural actions on this space. 

\begin{itemize}
\item An action of ${\rm GL}_2$ given by \[A\cdot f(x_0,x_1) = {\rm det}(A)^{n/2-1}f(A^{-1}(x_0,x_1)).\]
\item An action of ${\rm PGL}_2 \times \Gm$ given by \[(\left[A\right],t)\cdot f(x_0,x_1)={\rm det}(A)^{n/2}t^{-2}(f(A^{-1}(x_0,x_1))).\]
\end{itemize}

The subset of square-free forms inside $\bA^{n+1}$ is open and invariant with respect to both actions. We will denote it by $\bA^{n+1}_{\rm sm}$.

\begin{thm}[\cite{ArsVis}*{Cor. 4.7}]\label{thm:presentation}
When $g\geq 2$ is even, we have an isomorphism $\left[\bA^{2g+3}_{\rm sm}/{\rm GL}_2\right] \simeq \Hcal_g$.

When $g\geq 3$ is odd, we have an isomorphism $\left[\bA^{2g+3}_{\rm sm}/{\rm PGL}_2\times \Gm\right] \simeq \Hcal_g$.
\end{thm}

When no confusion is possible, we will write $G$ for either ${\rm GL}_2$ or ${\rm PGL}_2\times \Gm$. Our computation will be for the most part done on the projectivizations \[\PP^n = \left( \bA^{n+1}\smallsetminus\{0\} \right)/\Gm\] equipped with the induced action of $G$. 

There is a $G$-invariant stratification on $\PP^n$. Let $\Delta_{i}^n \subset \PP^n$ be the closed subscheme of $\PP^n$ given by forms of degree $n$ divisible by the square of a form of degree at least $i$, with the reduced subscheme structure. Then 
\[ \PP^n \supset \Delta^n_1 \supset \ldots \supset \Delta_{n/2}^n \]
is a $G$-invariant stratification of $\PP^n$. We set $\PP^n_{\rm sm}:= \PP^n\smallsetminus \Delta_1^n$.

There is a natural equivariant map $\PP^{n-2r} \times \PP^r \rightarrow \Delta^{n}_{r}$ given by $(f,g) \rightarrow fg^2$. If we restrict the map to $\PP^{n-2r}_{\rm sm} \times \PP^r$ then the image is exactly $\Delta_{r}^n \setminus \Delta_{r+1}^n$.

\begin{prop}[\cite{PirCohHypEven}*{Prop. 3.3}]\label{prop:homeo}
The map $\PP^{n-2r}_{\rm sm} \times \PP^r \rightarrow \Delta^n_{r}\setminus \Delta^n_{r+1}$ is an equivariant universal homeomorphism.
\end{prop}

\subsection{Cohomological invariants and Chow groups with coefficients}

In this Subsection we describe some properties of cohomological invariants and (equivariant) Chow groups with coefficients which will be needed throughout the paper.

First, given a morphism $f:\Xcal \to \mathcal{Y}$, we have a pullback
\[
f^*:\Inv(\Xcal,{\rm H}_{\ell}) \to \Inv(\mathcal{Y},{\rm H}_{\ell})
\]
given by $f^*(\alpha)(p)=\alpha(f(p))$. Note that given an algebraic stack $\Xcal$ over $k$, the pullback
\[
\Het_{\ell}(k)=\Inv({\rm Spec}(k),{\rm H}_{\ell}) \to \Inv(\Xcal,{\rm H}_{\ell})
\]
induces a $\Het_{\ell}(k)$-module structure on $\Inv(\Xcal,{\rm H}_{\ell})$. This map need not be injective, but it is so if $\Xcal$ has a rational point, which will always be true for the stacks we consider in this paper. Throughout the paper, we will see our rings of cohomological invariants as $\Het_{\ell}(k)$-algebras and the presentations we give will always be as $\Het_{\ell}(k)$-modules.

A smooth-Nisnevich morphism \cite{PirAlgStack}*{Def. 3.2} is a representable smooth morphism $f:\mathcal{Y} \to \Xcal$ of algebraic stacks such that for any point $p:\Spec(F) \to \Xcal$ we have a lifting
\[
\xymatrix{  & \mathcal{Y} \ar[d]^f \\
        \Spec(F) \ar[ur]^{p'} \ar[r]^{p} & \Xcal  }
\]

Using the definition of pullback one can easily conclude that the pullback through a smooth-Nisnevich morphism is injective. In fact, much more is true:

\begin{thm}[\cite{PirAlgStack}*{Thm. 3.8}]\label{thm:sheaf}
The functor $\Inv(-,{\rm H}_{\ell})$ is a smooth-Nisnevich sheaf.
\end{thm}

Let $G$ be a special smooth algebraic group, in the sense that exery $G$-torsor over a scheme is Zariski-locally trivial.
Then the quotient map $X \to \left[ X/G\right]$ is an example of a smooth-Nisnevich morphism. In particular, this includes the presentation $\bA^{2g+3}_{\rm sm} \to \left[ \bA^{2g+3}_{\rm sm}/{\rm GL}_2 \right]=\Hcal_g$ for even $g$, but not $\bA^{2g+3}_{\rm sm} \to \left[ \bA^{2g+3}/{\rm PGL}_2\times \Gm \right]=\Hcal_g$ for odd $g$, and the latter is indeed not a smooth-Nisnevich morphism, and we will see later (Remark \ref{rmk: non-sN})  that the pullback on cohomological invariants is not injective.

As the following Proposition shows, cohomological invariants are determined by their ``generic'' value:

\begin{prop}\label{prop:generic}
Let $\Xcal$ be a smooth, connected algebraic stack, and let $\mathcal{U} \xrightarrow{j} \Xcal$ be an open immersion.  Then the pullback $j^*:\Inv(\Xcal,{\rm H}_{\ell}) \to \Inv(\mathcal{U},{\rm H}_{\ell})$ is injective. 

Moreover, let $X$ be a smooth, irreducible scheme of finite type over $k$ with generic point $\xi$, and let $\alpha \in \Inv(X,{\rm H}_{\ell})$. Then
\[ \alpha = 0 \Leftrightarrow \alpha(\xi)=0. \]
\end{prop}
\begin{proof}
The second statement is \cite{PirAlgStack}*{Cor. 4.7}. We want to show that it implies the first. This is immediate if $\Xcal$ is a scheme.

In general, assume that $j^*\alpha=0$. Let $X \to \Xcal$ be a smooth-Nisnevich covering by a scheme. These always exist thanks to \cite{PirAlgStack}*{Prop. 3.6}. We can take the pullback of $\mathcal{U}$ along $X\to\Xcal$, obtaining an everywhere dense subset $U \subset X$ on which $\alpha$ is zero. This implies that $\alpha$ is zero when restricted to $X$, which in turn shows that $\alpha =0$.
\end{proof}

As mentioned in the Introduction, for a smooth scheme $X$, by \cite{PirAlgStack}*{Prop. 4.9} we have 
\[
\Inv(X,{\rm H}_{\ell})=A^0(X,{\rm H}_{\ell})
\]
where the group on the right hand side is the zero-codimensional equivariant Chow group with coefficients \cite{Rost}. 

If $X$ is an equidimensional scheme over $k$ its $i$-th codimensional Chow group with coefficients $A^i(X,{\rm H}_{\ell})$ is the $i$-th homology group of the complex
\[
0 \to C^0(X,{\rm H}_{\ell}) \to C^1(X,{\rm H}_{\ell}) \to \ldots \to C^{{\rm dim} X}(X,{\rm H}_{\ell}) \to 0
\]
The group $C^i(X,{\rm H}_{\ell})$ is generated by elements $(Z,a_Z)$, where $Z$ is an irreducible subvariety of codimension $i$ and $a_Z \in \Het_{\ell}(k(Z))$. The differential decreases the cohomological degree by one. In particular, $A^0(X,{\rm H}_{\ell})$ is the subset of unramified elements of $\Het_{\ell}(k(X))$.

Chow groups with coefficients have all the same properties of ordinary Chow groups:
\begin{itemize}
    \item Pullback $f^*$ for maps that are flat or whose target is smooth.
    \item Pushforward $f_*$ for proper maps.
    \item Projective bundle formula.
    \item Chern classes.
    \item Ring structure for smooth schemes.
\end{itemize}

The pullback on $A^0(X,{\rm H}_{\ell})$ is compatible with the pullback on cohomological invariants. Moreover, given an irreducible subvariety $V\subset X$ of codimension $r$ there is a localization long exact sequence 
\[
\ldots \to A^i(X,{\rm H}_{\ell}) \to A^i(U,{\rm H}_{\ell}) \xrightarrow{\partial} A^{i+1-r}(V,{\rm H}_{\ell}) \to A^{i+1}(X,{\rm H}_{\ell}) \to \ldots
\]
where the boundary map $\partial$ lowers the cohomological degree by $1$, and it is compatible with Chern classes, pullbacks and pushforwards. 

There is an additional property of the boundary map $\partial$ that will be crucial to our results:
\begin{lm}\label{lm:deriv}
Assume $X,V$ as above are both smooth, and $V$ has codimension $1$. 

Let $\alpha$ be an element of $A^0(X,{\rm H}_{\ell})$, and let $f \in \mathcal{O}(U)^*$ be a local equation for $V$. Then $\lbrace f \rbrace \in k(U)^*/(k(U)^*)^\ell = {\rm H}^1_{\ell}(k)$ is a class in $A^0(U,{\rm H}_{\ell})$.  We have
\[
\partial(f \cdot \alpha)=\alpha_{\mid V}.
\]
\end{lm}
\begin{proof}
This is a direct consequence of the definitions of the boundary map $\partial_V$ \cite{Rost}*{3.7} and the map $s_v^{\pi}$ \cite{Rost}*{p. 328}, together with \cite{Rost}*{Cor. 12.4}.
\end{proof}

The main source for Chow groups with coefficients is Rost's original paper \cite{Rost}. A recap of the theory, as well as a theory of Chern classes, which is not present in Rost's paper, can be found in \cite{PirCohHypEven}*{Sec. 2} or \cite{DilCohHypOdd}*{Sec. 1}.

Using the identification with Chow groups with coefficients and the sheaf condition we obtain:

\begin{prop}\label{Prop:homot inv}
Let $f:\Xcal \to \mathcal{Y}$ be a morphism of smooth algebraic stacks. If $f$ is either
\begin{itemize}
    \item an open immersion whose complement has codimension $>1$,
    \item an affine bundle,
    \item the projectivization of a vector bundle,
\end{itemize}
then $f^*$ is an isomorphism on cohomological invariants.
\end{prop}
\begin{proof}
The first statement is \cite{PirAlgStack}*{Prop. 4.13}. The second statement is \cite{PirAlgStack}*{Prop. 4.14}. To prove the third statement, let $\mathcal{V} \to \Xcal$ be a vector bundle of rank $\geq 2$, and let $0_{\mathcal{V}} : \mathcal{X} \to \mathcal{V} $ be its zero section. Consider the factorization 
\[
\mathcal{V} \setminus 0_{\mathcal{V}} \to \PP(\mathcal{V}) \xrightarrow{f} \mathcal{X}.
\]
By the first two properties, the pullback from $\Inv(\mathcal{X},{\rm H}_{\ell})$ to $\Inv(\mathcal{V} \setminus 0_{\mathcal{V}}, {\rm H}_{\ell})$ is an isomorphism. On the other hand $\mathcal{V} \setminus 0_{\mathcal{V}}$ is an open subscheme in the total space of $\mathcal{O}(-1)$, the tautological line bundle over $\PP(\mathcal{V})$: by Proposition \ref{prop:generic} and the second property the pullback from $\Inv(\PP(\mathcal{V}),{\rm H}_{\ell})$ to $\Inv(\mathcal{V} \setminus 0_{\mathcal{V}}, {\rm H}_{\ell})$ is injective, hence $f^*$ must be an isomorphism.
\end{proof}

As an easy application of the properties above, we get the following computation which will be used later.

\begin{cor}\label{cor:Gm}
Let $X$ be a smooth scheme. Then 
\[ 
A^0(X\times \Gm,{\rm H}_{\ell}) = A^0(X,{\rm H}_{\ell}) \oplus \lbrace t \rbrace \cdot A^0(X\times \Gm,{\rm H}_{\ell})
\]
Where $\Gm={\rm Spec}(k\left[ t,t^{-1}\right])$ and $\lbrace t \rbrace \in {\rm H}^1_{\ell}(k(x))=k(t)^*/(k(t)^*)^{\ell}$.
\end{cor}
\begin{proof}
Consider the inclusion $X \to X\times \bA^1$ given by the zero section. It induces the localization exact sequence
\[
0 \to A^0(X\times \bA^1,{\rm H}_{\ell}) \to A^0(X \times \Gm, {\rm H}_{\ell}) \xrightarrow{\partial} A^0(X, {\rm H}_{\ell}).
\]
Note that by Proposition \ref{Prop:homot inv} we have $A^0(X\times \bA^1,{\rm H}_{\ell})=A^0(X,{\rm H}_{\ell})$. Now, given an element $\alpha \in A^0(X,{\rm H}_{\ell})$, the element $\lbrace t \rbrace \cdot \alpha \in A^0(X \times \Gm, {\rm H}_{\ell})$ satisfies $\partial(\lbrace t \rbrace \cdot \alpha) = \alpha$ by Lemma \ref{lm:deriv}. This shows that $\partial$ is a split surjection, allowing us to conclude.
\end{proof}

Given a smooth scheme $X$ being acted upon by a smooth affine algebraic group $G$, there is a standard procedure to produce a scheme $X_i$ whose cohomology, Chow groups, etc. are the same as those of the quotient stack $\left[ X/G\right]$ up to a certain degree $i$. This was first done in order to construct equivariant Chow groups by Edidin-Graham and Totaro \cites{EG, Tot} and it was first used for Chow groups with coefficients by Guillot \cite{Guil}.

Let $V_i$ be a representation of $G$ which is free outside of a closed subset of codimension greater than $i+1$. We can always find such a representation as every affine group scheme over a field is linear. Let $U_i$ be the subset on which the action is free. We define
\[ X_i = (U_i\times X)/G. \]
As the action is free, the quotient $X_i$ is an algebraic space. We will often refer to this construction as an \emph{equivariant approximation} of $[X/G]$. Moreover, we can always pick a suitable representation so that the quotient is a scheme \cite{EG}*{Lemma 9}. We define
\[ A^i_G(X,{\rm H}_{\ell}) \overset{\rm def}{=} A^i(X_i,{\rm H}_{\ell}). \]
This group is well defined thanks to a simple double fibration argument (see \cite{EG}*{Prop. 1}), and all the properties of ordinary Chow groups with coefficients extend to the equivariant counterpart. By \cite{PirCohHypEven}*{2.10} we have the equality
\[
\Inv(\left[ X/G \right],{\rm H}_{\ell})=A^0_G(X,{\rm H}_{\ell}).
\]

The following simple computations will be needed later. These three statements are proven respectively in \cite{PirCohHypEven}*{2.11}, \cite{PirCohHypThree}*{Prop. 15} and \cite{PirCohHypThree}*{Prop. 14}, but for the sake of self-containment we include a short proof.

\begin{lm}\label{lm:GL2PGL2}
We have:
\begin{enumerate}
    \item For any $\ell$ there is an isomorphism $\Inv(\Brm {\rm GL}_2,{\rm H}_{\ell})\simeq\Het_{\ell}(k)$.
    \item If $\ell\neq 2$ the cohomolgical invariants of $\Brm {\rm PGL}_2$ with coefficients in ${\rm H}_{\ell}$ are trivial. If $\ell=2$ then $\Inv(\Brm {\rm PGL}_2)\simeq\Het_2(k) \oplus w_2 \cdot \Het_2(k)$, where $w_2$ is a cohomological invariant of degree $2$.
    \item Let $X$ be a scheme acted upon by a smooth affine algebraic group $H$, and additionally let $\Gm$ act trivially on it. Then $A^0_H(X,{\rm H}_{\ell})=A^0_{H \times \Gm}(X,{\rm H}_{\ell})$.
\end{enumerate}
\end{lm}
\begin{proof}
The first statement is an immediate consequence of the group ${\rm GL}_2$ being special: indeed, the latter property means that every ${\rm GL}_2$-torsor is Zariski locally trivial, hence the ${\rm GL}_2$-torsor ${\rm Spec}(k) \to \Brm {\rm GL}_2$ has a section or, in other terms, is a smooth-Nisnevich cover. By Theorem \ref{thm:sheaf} this implies \[\Inv(\Brm {\rm GL}_2, {\rm H}_{\ell}) \subset \Het_{\ell}(k).\]
On the other hand, the factorization ${\rm Spec}(k) \to \Brm {\rm GL}_2 \to {\rm Spec}(k)$ implies the opposite inclusion, so the two groups are equal.

To prove the second statement, let $U_n \subset {\rm GL}_n$ be the subscheme of symmetric invertible matrices, and let $\Gm^n \subset U_n$ be the subscheme of invertible diagonal matrices. Consider the commutative diagram:

\[ \xymatrix{
			\Gm^n \ar[r]^{i} \ar[d] & U_n \ar[d] \\
			\Brm\mu_2^n \ar[r] & \Brm{\rm O}_n } \]
			
The vertical maps are given respectively by the quotient by $\Gm^n$ acting on itself with weight two and the quotient by ${\rm GL}_n$ acting by $(A, S) \rightarrow A^{\rm T} S A$. In particular we can see the action of $\Gm^n$ on itself as the subgroup of diagonal matrices of ${\rm GL}_n$ acting on $\Gm^n \subset U_n$. The bottom map comes from the inclusion of the diagonal matrices with coefficients $\pm 1$ into ${\rm O}_n$. Note that both vertical maps are quotients by special groups and thus smooth-Nisnevich.

It's a well known fact that in characteristic different from two every symmetric matrix is equivalent to a diagonal matrix under the action of ${\rm GL}_n$ defined above. An immediate consequence of this fact is that the map $\Gm^n\to \Brm {\rm O}_n$ is smooth-Nisnevich, hence $\Brm\mu_2^n\to \Brm{\rm O}_n$ is smooth-Nisnevich too.

Now, \cite{PirCohHypThree}*{Prop. 4, Cor. 7} shows that the cohomological invariants of $\Brm \mu_2$ are trivial for $\ell \neq 2$, and moreover that for any $\ell$ the invariants of $\Xcal \times_k \Brm \mu_2$ are equal to $\Inv(\Xcal, {\rm H}_{\ell}) \otimes_{\Het_{\ell}(k)} \Inv(\Brm \mu_2, {\rm H}_{\ell})$. As $\Brm \mu_2^n = (\Brm \mu_2)^{\times_k n}$, we have just shown the cohomological invariants of $\Brm {\rm O}_n$ are trivial for $\ell \neq 2$.

In characteristic different from $2$ we have an isomorphism ${\rm PGL}_2 \simeq {\rm SO}_3$, and ${\rm O}_3 = {\rm SO}_3 \times \mu_2$. Then $\Brm {\rm O}_3 = \Brm {\rm SO}_3 \times \Brm \mu_2$ and again by the formula above we conclude that any nontrivial cohomological invariant of $\Brm {\rm SO}_3=\Brm {\rm PGL}_2$ is of $2$-torsion. The description of the mod $2$ cohomological of $\Brm {\rm SO}_n$ can be found in \cite{GMS}*{Thm. 19.1}.

To prove the third statement, let $X' = (U \times X)/H$ be an equivariant approximation of $\left[X/H\right]$ such that $U$ is the open subset of a representation $V$ of $H$ where the action is free and the codimension of the complement of $U$ is two or more. Then $\Inv(X',{\rm H}_{\ell})=\Inv(\left[X/H\right]) $. Now consider the multiplicative action of $\Gm$ on $\bA^2$. This action is free on $\bA^2 \setminus 0$. Then $X''=(X \times U \times (\bA^2 \setminus 0))/H\times \Gm$ is an equivariant approximation of $\left[ X / H\times \Gm\right]$ and $\Inv(X'',{\rm H}_{\ell})=\Inv(\left[ X / H\times \Gm\right],{\rm H}_{\ell})$. On the other hand, one immediately sees that $X'' = X' \times \PP^1$, so we can conclude using Proposition \ref{Prop:homot inv}.
\end{proof}

\begin{rmk}\label{rmk: non-sN}
Conisder the morphism $\bA^{2g+3}_{\rm sm} \to \left[ \bA^{2g+3}_{\rm sm}/{\rm PGL}_2 \times \Gm \right]$. We claimed earlier that this morphism is not smooth-Nisnevich. Note that the action of ${\rm PGL}_2 \times \Gm$ extends to $\bA^{2g+3}$. Then we have a factorization
\[ \bA^{2g+3}_{\rm sm} \to \bA^{2g+3} \to {\rm B}({\rm PGL}_2\times \Gm)\]
as the invariants of $\bA^{2g+3}$ are trivial, while those of ${\rm B}({\rm PGL}_2\times \Gm)$ are not, the map cannot be injective. In particular, the map is not smooth Nisnevich. On the other hand, the map
\[
\left[ \bA^{2g+3}_{\rm sm}/{\rm PGL}_2\times \Gm\right] \to {\rm B}({\rm PGL}_2\times \Gm)
\]
is an open subset of a vector bundle, so the pullback on cohomological invariants is injective. This shows that 
\[ \bA^{2g+3}_{\rm sm} \to \left[ \bA^{2g+3}_{\rm sm}/{\rm PGL}_2\times \Gm\right] = \Hcal_g 
\]
cannot be smooth-Nisnevich, as the pullback on cohomological invariants is not injective.
\end{rmk}

One last thing that we will use is the following:

\begin{prop}[\cite{PirCohHypEven}*{Prop. 3.4}]\label{prop:homeo iso}
Let $X,Y$ be schemes over $k$. Assume $f:X\rightarrow Y$ is an (equivariant) universal homeomorphism. Then $f_*$ induces an isomorphism on (equivariant) Chow groups with coefficients.
\end{prop}

\section{Cohomological invariants from Weierstrass divisors}\label{sec:InvWei}
We recall some basic notions on families of hyperelliptic curves. A more detailed discussion can be found in \cite{KL}. A family of hyperelliptic curves $C \rightarrow S$ of genus $g$ is defined as a proper and smooth morphism whose fibres are curves of genus $g$, and moreover there exists a \emph{hyperelliptic involution} $\iota: C \rightarrow C$ of $S$-schemes such that the quotient $C' := C/\iota$ is a family of conics over $S$.

The ramification divisor of the projection $C \rightarrow C'$, equipped with the scheme structure given by the zeroth Fitting ideal of $\Omega_{C/C'}$, is called the \emph{Weierstrass subscheme} of $C/S$, and it is denoted $W_{C/S}$. The morphism $W_{C/S} \rightarrow C'$ is a closed immersion, so we will use the same notation for the divisor on $C$ and on $C'$ when no confusion is possible. The scheme $W_{C/S}$ is finite and \'etale over $S$ of degree $2g + 2$. 

The functor sending a family $C/S$ to (the spectrum of) its Weierstrass subscheme $W_{C/S}$ defines a morphism from $\Hcal_g$ to $\'Et_{2g+2}$, the stack of \'etale algebras of degree $2g+2$, which is in turn isomorphic to the classifying stack $\Brm {\rm S}_{2g+2}$ of ${\rm S}_{2g+2}$-torsors.

More generally, consider $\bA^{n+1}$ as the space of binary forms of degree $n$, and let $\bA^{n+1}_{\rm sm}$ be the open subset of non degenerate forms. Then there is a morphism $\bA^{n+1}_{\rm sm} \rightarrow \'Et_{n}$ obtained by sending a form $f$ to the zero locus $V_f \subset \PP^1$. This map factors through the projectivization $\PP^{n}_{\rm sm}$. 

When $n=2g+2$, the map $\bA^{2g+3}_{\rm sm} \rightarrow \'Et_{2g+2}$ factors through Arsie and Vistoli's presentation (Theorem \ref{thm:presentation}) 
\[\bA^{2g+3}_{\rm sm}\rightarrow \Hcal_g \rightarrow \'Et_{2g+2}.\]
In fact, if we pull back the universal family $\mathcal{C}_g \rightarrow \Hcal_g$ to $\bA^{2g+3}_{\rm sm}$ we see that given a morphism $S \xrightarrow{f} \bA^{2g+3}_{\rm sm}$ we obtain a family $C_{f}/S$ such that $C_{f}/\iota = \PP^1_S$ and $V_f = W_{C_f/S}$.

\begin{prop}\label{prop: sm Nis}
Let $n$ be an even positive number. The morphism $\bA^{n+1}_{\rm sm} \rightarrow \Brm {\rm S}_{n}$ is smooth-Nisnevich. In particular, the pullback $\Inv(\Brm {\rm S}_{n},{\rm H}_{\ell}) \to \Inv(\bA^{n+1}_{\rm sm},{\rm H}_{\ell})$ is injective.
\end{prop}
\begin{proof}
Write down a form of degree $n$ as $f(\lambda_1, \lambda_2)=x_0 \lambda_1^n + x_1 \lambda_1^{n-1}\lambda_2 + \ldots + x_n\lambda_2^n$.
Consider the subscheme $V \subset \bA^{n+1}_{\rm sm}$ given by $x_0=1$. Denote by $\Delta$ the subset of $\bA^n$ where the coordinates are not distinct. We have a map $(\bA^n \setminus \Delta) \rightarrow V$ given by $(\alpha_1, \ldots, \alpha_n) \rightarrow (\lambda_1 + \alpha_1\lambda_2) \ldots (\lambda_1 + \alpha_n \lambda_2)$. 

This map is clearly the ${\rm S}_n$-torsor inducing the map $V \rightarrow \Brm {\rm S}_{n}$. As the action of ${\rm S}_n$ on $\bA^n \setminus \Delta$ is free the torsor is versal \cite[5.1-5.3]{GMS}, which implies our claim.\footnote{When $k$ is finite it is smooth $\ell$-Nisnevich for any $\ell$, see \cite{PirAlgStack}*{Def. 3.4, Lm. 3.5}}
\end{proof}

In general, given a smooth stack $\Xcal$ and a factorization $\bA^{n+1}_{\rm sm} \xrightarrow{\phi} \Xcal \xrightarrow{\pi} \Brm {\rm S}_{n}$ the pullback $\pi^*$ on cohomological invariants is injective, as $(\pi \cdot \phi)^*=\phi^* \cdot \pi^*$ is injective by Proposition \ref{prop: sm Nis}.

\begin{cor}
The pullback $\Inv(\Brm {\rm S}_{2g+2},{\rm H}_{\ell}) \rightarrow \Inv(\Hcal_g,{\rm H}_{\ell})$ is injective.

Moreover, for all even $n$ the pullback $\Inv(\Brm {\rm S}_n,{\rm H}_{\ell}) \to \Inv(\left[\PP^n_{\rm sm}/G\right],{\rm H}_{\ell})$ is injective.
\end{cor}

A complete description of the cohomological invariants of $\Brm {\rm S}_n$ can be found in \cite[CH. VII]{GMS}. The invariants are trivial for $\ell \neq 2$, so we will concentrate on the case $\ell=2$. We briefly recall here some of their properties, in particular the ones that will be relevant for our work.

Let $E$ be an \'etale algebra over a field $K$ of degree $n$. We denote $m_x:E\to E$ the multiplication morphism by an element $x$ of $E$.
We can then define a morphism of classifying stacks \[\varphi:\Brm {\rm S}_n\longrightarrow \Brm {\rm O}_n\] by sending an \'etale algebra $E$ to the non-degenerate quadratic form on $E$ defined by the formula $x\mapsto {\rm Tr}(m_{x^2})$. Let $\alpha_i$ be the degree $i$ cohomological invariant obtained by pulling back the $i^{\rm th}$ Stiefel-Whitney class along $\varphi$. 

As a $\Het_2(k)$-module, $\Inv(\Brm {\rm S}_n)$ is freely generated by 
\[\alpha_0=1, \alpha_1, \ldots, \alpha_{\left[ n/2 \right]},\]
 where the degree of $\alpha_i$ is $i$.

Before proceeding further, let us explain how to explicitly compute the value of the cohomological invariants $\alpha_i$. 

As already said, we can associate to $E$ a quadratic form as follows: given an element $x$ of $E$, we define the quadratic form
\[ q_E(x):= {\rm Tr}(m_{x^2}) \]
Regarding $E$ as a vector space of dimension $n$, choose a basis $e_1,\dots,e_n$ of $E$ such that $q_E(x)=\sum_{i=1}^n \lambda_i x_i^2$, where $x=x_1 e_1+\dots + x_n e_n$. The $\lambda_i$ belong to $K^*$ as the form is nondegenerate. If $\sigma_i$ denotes the elementary symmetric polynomial of degree $i$ in $n$ variables, we have:
\[ \alpha_i(E)= \sigma_i(\lbrace \lambda_1 \rbrace, \dots, \lbrace \lambda_n \rbrace) \in {\rm H}_2^{i}(K) \]
where $\lbrace \lambda_j \rbrace$ are the corresponding classes in ${\rm H}_2^1(K)\simeq K^*/(K^*)^2$ and the product is the one defined in cohomology.

The multiplicative structure of the invariants $\alpha_i$ is described the following way (see \cite{GMS}*{Remark 17.4}). Given $s,r\leq \left[ n/2 \right]$, write $s=\sum_{i \in S} 2^i, r = \sum_{i \in R} 2^i$ and let $m = \sum_{i \in S\cap R} 2^i$. Then \[\alpha_s \cdot \alpha_r = \lbrace -1 \rbrace^m \cdot \alpha_{r+s-m}.\]

Let $E$ denote an \'etale algebra over of degree $n$ over an extension $K/k$ and write $\alpha_{\rm tot}= \sum_i \alpha_i$. By definition the invariants $\alpha_i$ take values in $\Het_2(K)$, and the following properties hold:

\begin{enumerate}

\item $\alpha_i(E)=0$ if $i > [n/2]+1$.
\item $\alpha_{[n/2]+1}(E)=\lbrace 2 \rbrace \cdot \alpha_{[n/2]}$ if $[n/2]+1$ is even, and $0$ otherwise.
\item $\alpha_{\rm tot}(K^{\oplus n})=1$.
\item $\alpha_{\rm tot}(K\left[x \right]/(x^2 - a))=1+\lbrace a \rbrace$.
\item $\alpha_{\rm tot}(E \times E') = \alpha_{\rm tot}(E)\alpha_{\rm tot}(E')$.

\end{enumerate}





We can now move on to our first computation:

\begin{cor}\label{cor:inv Pn}
	Let $n\geq 0$ be even and let $G$ be either ${\rm GL}_2$, ${\rm PGL}_2 \times \Gm$ or the trivial group. Then the last morphism of the exact sequence 
	\[ 0 \rightarrow A^0_G(\PP^n) \rightarrow A^0_G(\PP^n_{\rm sm}) \xrightarrow{\partial} A^0_G(\Delta^n_1)\]
 	is surjective. Moreover:
 
 \begin{itemize}
     \item If $G$ is equal to ${\rm GL}_2$ or trivial, the inclusion $\Inv({\rm S}_{n/2}) \subset \Inv(\left[ \PP^n_{\rm sm} /G \right])$ is an isomorphism. 
     \item 	If $G={\rm PGL}_2$ the cokernel of the inclusion is a free $\Het_2(k)$-module generated by the invariant $w_2$ coming from the cohomological invariants of $\Brm {\rm PGL}_2$.
     \item The pullback $\Inv(\Hcal_g) \to \Inv(\PP^{2g+2}_{\rm sm})$ is an isomorphism when $g$ is even, and when $g$ is odd it is surjective with kernel the $\Het_2(k)$-module generated by $w_2$.
 \end{itemize}

\end{cor}
\begin{proof}
Let us first assume $G={\rm GL}_2$ or $G$ is trivial. The proof is identical. We proceed by induction on the even integer $n$, the case $n=0$ being trivial. By \cite[3.3, 3.4]{PirCohHypEven} we know that $A^0_G(\Delta_1^n) \simeq A^0_G(\PP^{n-2}_{\rm sm}\times \PP^1)$, which by the inductive hypothesis and the projective bundle formula is isomorphic to $\Inv(\Brm {\rm S}_{n/2})$, which is freely generated as a $\Het_2(k)$-module by the Stiefel-Whitney classes $1, \alpha_1, \ldots, \alpha_{\left[ (n-2)/2\right]}$.

Note that  and $A^0_G(\PP^n)$ is trivial as $\left[\PP^n/G\right]$ is a projective bundle over $\Brm G$, and $\Inv(\Brm {\rm GL}_2)=\Inv({\rm Spec}(k)) = \Het_2(k)$ by the first point of Lemma \ref{lm:GL2PGL2}. Using the fact that $A^0_G(\PP^n_{\rm sm})$ has to contain the cohomological invariants of $\Brm {\rm S}_{n}$, we see that the cokernel of $A^0_G(\PP^n) \rightarrow A^0_G(\PP^n_{\rm sm})$ must contain a free module generated by elements $y_1, \ldots, y_{\left[ n/2\right]}$, of degree $\mathrm{deg}(y_i)=i$.

The map \[\del:A^0_G(\PP^n_{\rm sm}) \longrightarrow A^0_G(\Delta_1^n)\] lowers degree by one. As the images of the $y_i$ must be linearly independent, we must have that $\del(y_1)=1$, $\del(y_2)=a_{2,1} + \alpha_1$ with $a_{2,1} \in {\rm H}_2^1(k)$, and in general $\del(y_i) = \sum_{j<i-1} a_{i,j} \alpha_j + \alpha_{i-1}$ with $a_{i,j} \in {\rm H}_2^{i-j-1}(k)$. It follows immediately that it must be surjective. An easy consequence of this is that there can be no other element in $A^0_G(\PP^n_{\rm sm})$, otherwise the kernel of $\del$ would have to be bigger than $A^0_G(\PP^n)=\Het_2(k)$.

Now assume that $G={\rm PGL}_2\times \Gm$.

Again, the map from $\left[\PP^n/G\right]$ to $\Brm G$ induces an isomorphism on cohomological invariants, being a projective bundle. In particular, by the second and third point of Lemma \ref{lm:GL2PGL2}, there is a submodule of $A^0_G(\PP^n_{\rm sm})$ that is isomorphic to $\Het_2(k) \cdot w_2$, and this submodule is in the kernel of $\del$. Now let us prove that $\Het_2(k) \cdot w_2 \cap \langle 1, \alpha_1, \ldots \alpha_{n/2} \rangle = 0$. By \cite[Cor. 19]{PirCohHypThree} if we take the pullback 
\[
\Inv(\left[ \PP^n_{\rm sm} /G \right]) \to \Inv(\left[ (\PP^n_{\rm sm}\times \PP^1)/G\right])
\]
the element $w_2$ maps to zero, because $\left[ (\PP^n\times \PP^1)/G)\right]$ is a projective bundle over $\left[ \PP^1/G\right]$, whose invariants are trivial by \cite{PirCohHypThree}*{Prop. 16}, while the map is injective on the elements coming from $\Inv(\Brm {\rm S}_n)$ as the composition 
\[
\PP^1 \times \PP^n_{\rm sm} \to \PP^n_{\rm sm} \to \Brm {\rm S}_n
\]
is still smooth-Nisnevich. This proves our claim. 

If $n=2$, $\Delta^2_{1}$ is universally homeomorphic to $\PP^1$, so by \cite{PirCohHypThree}*{Prop. 16} $A^0_G(\Delta^2_1)$ is trivial and the pushforward $i_*A^0_G(\Delta^2_1) \to A^1_G(\PP^2)$ is zero \cite[Prop. 23]{PirCohHypThree}.

The requirement that the pushforward $ A^0_G(\Delta_{1,m}) \to A^1_G(\PP^{m})$ be trivial is implied by the surjectivity of $\del$, and as seen above it is true for $m=2$, so we can assume by induction that it is true for $m=n-2i, i > 0$. In particular, by \cite[Cor. 19]{PirCohHypThree}, we have that 
\[A^0_G(\PP^{n-2i}_{\rm sm} \times \PP^1)=A^0_G(\PP^{n-2i}_{\rm sm})/\langle w_2 \rangle\]
and the inductive hypothesis implies that
\[A^0_G(\PP^{n-2i}_{\rm sm})/\langle w_2 \rangle = \Inv(\Brm {\rm S}_{n-2i}).\]
We want to show that $A^0_G(\PP^n_{\rm sm}) \xrightarrow{\partial} A^0_G(\Delta_1^n)$ is surjective. By the reasoning above \[A^0_G(\Delta^n_1) \subset A^0_G(\Delta^n_1\setminus \Delta^n_2) \simeq A^0_G(\PP^{n-2}_{\rm sm} \times \PP^1)=\Inv(\Brm {\rm S}_{n-2}).\] 
We proved that the cokernel of the first map in the exact sequence contains a free $\Het_2(k)$-module generated by elements $y_1, \ldots, y_{n/2}$ of degree ${\rm deg}(y_i)=i$. The $\Het_2(k)$-module $\Inv(\Brm {\rm S}_{n-2})$ is freely generated by $1,\alpha_1,\ldots, \alpha_{n/2-1}$, so we are in the exact same situation as in the $G={\rm GL}_2$ case and the proof is identical. 

Finally, the comparison between $A^0(\PP^{2g+2}_{\rm sm})$ and $A^0_G(\PP^{2g+2}_{\rm sm})$ is immediate by the results above, the factorization 
\[
\PP^{2g+2}_{\rm sm} \to \left[ \PP^{2g+2}_{\rm sm} / G \right] \to {\rm BS}_n 
\]

and the fact that $w_2$ pulls back to zero when $g$ is odd.
\end{proof}

As the reader can see by comparing the proof above with the proofs of the same statements in \cite{PirCohHypEven}*{Sec. 4}, \cite{PirCohHypThree}*{Sec. 3} and \cite{DilCohHypOdd}*{Sec. 5-6}, having prior knowledge of a large subalgebra of the cohomological invariants of our stacks not only allows us to drop the condition of the base field being algebraically closed, but it also makes the proof substantially easier.

In the next Section we will explicitly construct another invariant of $\Hcal_g$ when $g$ is even.
This will allow us to conclude the generalization of the results of \cite{PirCohHypEven} in Section \ref{sec:mult}. 

\section{The last invariant}\label{sec:g even}
For the rest of the Section, we assume that $g$ is even. Consider the open subset $\overline{U}_0=\left\{x_0 \neq 0 \right\}$ inside $\PP^{2g+2}_{\rm sm}$, and let $U_0$ be its preimage in $\bA^{2g+3}_{\rm sm}$. The $\Gm$-torsor $U_0\to\overline{U}_0$ is clearly trivial. Consequently by Corollary \ref{cor:Gm} we have 
\[
A^0(U_0)=A^0(\overline{U}_0)\times A^0(\Gm) = A^0(\overline{U}_0) \oplus \tau \cdot A^0(\overline{U}_0)
\]
where $\tau$ is the cohomological invariant that sends a $K$-point $(x_0:x_1:\dots:x_{2g+2})$ to $\lbrace x_0 \rbrace$ in ${\rm H}_2^1(K)\simeq K^*/(K^*)^2$.
The multiplicative structure is defined by the single additional relation $\tau^2 = \lbrace -1 \rbrace \cdot \tau$, coming from the relation $\lbrace a, -a \rbrace =0$ in Milnor's $\rm K$-theory \cite{Rost}*{Sec. 1}.

The invariant $\tau$ clearly does not extend to a cohomological invariant of $\bA^{2g+3}_{\rm sm}$, but we claim that the element $\beta_{g+2} := \tau \cdot \alpha_{g+1}$ does.

\begin{prop}\label{split}
Let $g\geq 2$ be even. Then the element $\beta_{g+2}$ defined above extends to a cohomological invariant of $\bA^{2g+3}_{\rm sm}$. Moreover, $\beta_{g+2}$ is $\Het_2(k)$-linearly independent from the invariants coming from $\Brm {\rm S}_{2g+2}$.
\end{prop}
\begin{proof}
We have an exact sequence 
\[ 
0 \rightarrow A^0(\bA^{2g+3}_{\rm sm}) \rightarrow A^0(\overline{U}_0) \oplus \tau \cdot A^0(\overline{U}_0) \xrightarrow{\partial} A^0(V_0)
\]
where $V_0$ is the complement to $U_0$. We claim that the element $\tau \cdot \alpha_{g+1}$ maps to zero. As by Lemma \ref{lm:deriv} we have $\partial(\tau \cdot \alpha) = \alpha_{\mid V_0}$ for any $\alpha$ coming from $A^0(\bA^{2g+3}_{\rm sm})$, this is equivalent to saying that $\alpha_{g+1}$ becomes zero when restricted to $V_0$.

Consider the universal conic $\mathcal{C'}/\bA^{2g+3}_{\rm sm} \simeq \bA^{2g+3}_{\rm sm} \times \PP^1$. Restricting to the open subset $U_0$ is equivalent to requiring the Weierstrass divisor of a curve $C/S$ to not contain the divisor at infinity $S \times \infty$. Conversely, given a curve mapping to the complement $V_0$, the Weierstrass divisor will always have a section $S \rightarrow W_{C}$ given by $S \rightarrow S \times \infty$. In particular, given a field $K$ and a curve $C/K$ lying over $V_0$, the \'etale algebra $W_C/K$ will split as $W'_C \times K/K$.

Now we apply property $(5)$ of the Stiefel-Whitney classes. Looking at the part of degree $g+1$ we get
\[
\alpha_{g+1}(W'_C \times K)= \sum_{i+j=g+1} \alpha_i(W'_C) \cdot \alpha_{j}(K).
\]
By properties $(1), (2)$ the right hand side is zero, concluding our proof.

\end{proof}

Now we want to prove that this element glues to a cohomological invariant of $\Hcal_{g}$. We will show two different approaches to the problem. 

The first is a straight up computation that reduces the problem to a maximal torus inside ${\rm GL}_2$. The second is more subtle: we produce an invariant on a projective bundle over $\bA^{2g+3}_{\rm sm}$ which is trivially equivariant, but which we cannot a priori show to be nonzero. Then we show that after restricting to a locally closed subset it is equal to $\beta_{g+2}$, proving that it is independent from the invariants coming from $\Brm {\rm S}_{2g+2}$ (and in particular nonzero).

\subsection{First proof: reduction to the torus action}

\begin{lm}\label{CellDec}
Let $X \xrightarrow{f} Y$ be a map of algebraic spaces such that Zariski locally on $Y$ we have $X = Y \times Z$, where $Z$ is a smooth proper scheme admitting a cell decomposition $Z = \sqcup_{i\in I}(\sqcup_{j \in J_i} \bA^i)$. Then we have 
\[
A^{*}(X) \simeq A^{*}(Y) \otimes \mathrm{CH}^{\bullet}(Z).
\]
\end{lm}
\begin{proof}
We begin with the case where $X=Y \times Z$, proceeding by induction on the dimension of $Z$. Note that at this point we do not need the proper and smooth assumption on $Z$. 

If the dimension of $Z$ is zero, the statement is trivially true. Now let the dimension of $Z$ be equal to $n$, and let $Z' \subset Z$ be the union of all lower dimensional components, which is a closed subset of $Z$. For any $V \subseteq Z$ there is a map $A^{*}(Y) \otimes \mathrm{CH}^{\bullet}(V) \rightarrow A^{*}(X)$ given by $(a, b) \rightarrow a \times b$. We have a long exact sequence 
\[
\ldots \rightarrow A^s(Y \times Z) \rightarrow A^s(Y \times (\sqcup_{j \in J_n} \bA^n)) \xrightarrow{\partial} A^s(Y \times Z') \rightarrow A^{s+1}(Y \times Z) \rightarrow \ldots 
\]
As the Chow groups with coefficients of an affine bundle are isomorphic to those of the base we have $A^s(Y \times (\sqcup_{j \in J_n} \bA^n))\simeq A^s(Y) \otimes \ZZ^{\# J_n} \simeq A^s(Y) \otimes {\rm CH}^{\bullet}((\sqcup_{j \in J_n} \bA^n))$. Then we can conclude by comparing the long exact sequence above and the exact sequence
\[
\ldots \rightarrow \underset{i+j=s}{\oplus} A^i(Y) \otimes {\rm CH}^j(Z) \rightarrow \underset{i+j=s}{\oplus} A^i(Y) \otimes \ZZ^{\# J_n} \xrightarrow{\partial}\underset{i+j=s}{\oplus} A^i(Y) \otimes {\rm CH}^j(Z') \ldots  
\]
For the general case, note that we know the result to hold true for ordinary Chow groups \cite[Prop. 1]{EG97} (here is where the assumptions on $Z$ are needed). Thus we have a subring of $A^{*}(X)$ isomorphic to ${\rm CH}^{\bullet}(Z) \otimes \ZZ/2\ZZ$, and by taking multiplication this induces a map ${\rm CH}^{*}(Z) \otimes A^{*}(Y) \rightarrow A^{*}(X)$. 

Now let $U \subset Y$ be a Zariski open subset over which the fibration is trivial, and assume by induction that the formula holds on the complement $V$. The map ${\rm CH}^{\bullet}(Z) \otimes A^{*}(Y) \rightarrow A^{*}(X)$ is compatible with the isomorphisms $A^{*}(f^{-1}(V)) \simeq A^{*}(V) \otimes \mathrm{CH}^{\bullet}(Z)$, $A^{*}(f^{-1}(U)) \simeq A^{*}(U) \otimes \mathrm{CH}^{\bullet}(Z)$, so we can compare the two corresponding long exact sequences and conclude by the five lemma as above.
\end{proof}

\begin{prop}\label{Torus}
Let $G$ be a split, smooth, special algebraic group, and let $T\subseteq G$ be a maximal torus. Then for any $G$-scheme $X$ we have 
\[A^0_G(X) \simeq A^0_T(X).\]
\end{prop}
\begin{proof}
First, note that every special algebraic group is affine and connected \cite{Ser}*{Sec. 4.1}. After picking an equivariant approximation for $\left[X/G \right]$, we may assume that $\left[X/G \right]$ is an algebraic space. Let $T \subseteq B \subseteq G$ be a Borel subgroup. The map $\left[X /T \right] \rightarrow \left[ X/ G \right]$ factors as
\[
\left[X /T \right] \xrightarrow{f} \left[X /B \right] \xrightarrow{g} \left[ X/G \right]
\]
where the map $g$ is a $G/B$ bundle, and $f$ is a $B/T$ bundle. Both bundles are Zariski locally trivial due to the fact that $G$ is special. It's well known that $G/B$ is smooth, proper and admits a cell decomposition, while $B/T$ is an affine space, see for example \cite{MilGrp}*{Ch. 21}. Then we can use Lemma \ref{CellDec} and Proposition \ref{Prop:homot inv} to conclude. 
\end{proof}

\begin{rmk}
Lemma \ref{CellDec} and Proposition \ref{Torus} are true for any choice of $\ell$, not just $\ell=2$; it's easy to check that the proofs never use the specific choice of coefficients.
\end{rmk}

\begin{prop}\label{prop:glue}
Let $g\geq 2$ be an even integer. Then the cohomological invariant $\beta_{g+2}$ of Proposition \ref{split} glues to an invariant of $\Hcal_{g}$.
\end{prop}
\begin{proof}
By Proposition \ref{Torus}, we only need to prove that $\beta_{g+2}$ glues along the smooth-Nisnevich cover $\bA^{2g+3}_{\rm sm} \to \left[ \bA^{2g+3}_{\rm sm}/ \Gm^2\right] $. We have 
\[ \bA^{2g+3}_{\rm sm} \times_{\left[ \bA^{2g+3}_{\rm sm}/ \Gm^2\right]}  \bA^{2g+3}_{\rm sm}= \bA^{2g+3}_{\rm sm} \times \Gm^2, \]
and the two maps to $\bA^{2g+3}_{\rm sm}$ are respectively the first projection ${\rm Pr}_1$ and multiplication map $\rm m$. This is equivalent to asking that 
\[
{\rm Pr}_1^*\beta_{g+2} - {\rm m}^* \beta_{g+2}=0 \in \Inv(\bA^{2g+3}_{\rm sm} \times \Gm^2).
\]
By Proposition \ref{prop:generic} it suffices to check it on the generic point of $\bA^{2g+3}_{\rm sm}\times \Gm^2$. 

We identify $\Gm^2 = {\rm Spec}(k\left[\lambda_1, \lambda_2, (\lambda_1\lambda_2)^{-1}\right]$ and consequently 
\[
K\overset{\rm def}{=} k(\bA^{2g+3}_{\rm sm}\times \Gm^2)=k(x_0, \ldots, x_{2g+2}, \lambda_1, \lambda_2).
\]
Thus we have to prove that the difference
\[ 
{\rm Pr}_1^*( \lbrace x_0 \rbrace \cdot \alpha_{g+1}) - {\rm m}^*( \lbrace x_0 \rbrace \cdot  \alpha_{g+1})
\]
is zero as an element of $\Het_2(K)$. Here ${\rm m}$ is the multiplication map defining the action and ${\rm Pr}_1$ is the first projection. Note that we already know that $\alpha_{g+1}$ is invariant, so the question boils down to whether we have $(\lbrace x_0 \rbrace - {\rm m}^*\lbrace\lambda(x_0)\rbrace) \cdot \alpha_{g+1}=0$. 
Recall that ${\rm GL}_2$ acts by $A(f)={\rm det}(A)^{g}f(A^{-1})$, so in particular the map ${\rm m}^*$ sends $x_0$ to $(\lambda_1 \lambda_2)^g(\lambda_1)^{-g}x_0=\lambda_2^{-g}x_0$. As $g$ is even, we have 
\[\lbrace \lambda_2^{-g}x_0 \rbrace = \lbrace \lambda_2^{-g} \rbrace + \lbrace x_0 \rbrace = -g \lbrace \lambda_2 \rbrace + \lbrace x_0 \rbrace = \lbrace x_0 \rbrace,\]
concluding the proof.
\end{proof}

\subsection{Second proof: invariants of the universal conic}
$ \\ $

We keep assuming $g\geq 2$ even. Let $\Ccal'_g \rightarrow \Hcal_g$ be the universal conic bundle over $\Hcal_{g}$. It is the projectivization of a rank two vector bundle over $\Hcal_g$, so it has the same cohomological invariants. Pulling it back to $\bA^{2g+3}_{\rm sm}$, we obtain the ${\rm GL}_2$-equivariant projective bundle $C'_g = \PP^1 \times \bA^{2g+3}_{\rm sm} \rightarrow \bA^{2g+3}_{\rm sm}$. Its points can be seen as couples $(f,p)$ where $f$ is a non-degenerate form of degree $2g+2$ and $p\in \PP^1$. Consider a $K$-point $(p,f)$ on $C'_g$ such that $f$ is not zero at $p$, that is $p$ does not belong to the image of the Weierstrass divisor of the corresponding curve. Then $f(p)$ is well defined up to squares, so it defines an element in $K^{*}/(K^{*})^2 = {\rm H}_2^1(K)$. 

Let $U'$ be the ${\rm GL}_2$-equivariant open subset $\lbrace (p,f)\mid f(p) \neq 0 \rbrace$ of $C'_g$. The natural transformation $(p,f) \to f(p)$ defines a cohomological invariant on $U'$. This element clearly cannot extend to $C'_g$, but we claim that it does so after multiplication by $\alpha_{g+1}$:

\begin{prop}
The element $f \cdot \alpha_{g+1}$ is unramified on the universal conic over $\bA^{2g+3}_{\rm sm}$, and it glues to a cohomological invariant of $\mathcal{C}'_g$. Consequently it defines an invariant of $\Hcal_g$.
\end{prop}
\begin{proof}
Let $V'$ be the complement of $U'$ in $C'_g$. To show that the element extends, we need to check that the boundary map 
\[A^0(U'_g) \xrightarrow{\partial} A^0(V')\]
sends $f\cdot\alpha_{g+1}$ to zero.
As locally $f$ is precisely the equation for $V'$, by Lemma \ref{lm:deriv} we know that $\partial (f \cdot \alpha_{g+1}) = (\alpha_{g+1})_{\mid V'}$. Now we note that on $V'$ the Weierstrass divisor contains by definition a rational point, so $\alpha_{g+1}$ restricts to zero due to the same argument as Proposition \ref{split}.

To check the gluing condition along the smooth-Nisnevich cover $C'_g \to \mathcal{C}'_g$ we follow a reasoning similar to the one used in the proof of Proposition \ref{prop:glue}. We have
\[ C'_g \times_{\mathcal{C}'_g} C'_g = C'_g \times {\rm GL}_2 \]
and we claim that 
\begin{equation}\label{eq:invariance}
    {\rm Pr}_1^*(f\cdot\alpha_{g+1})-{\rm m}^*(f\cdot\alpha_{g+1})=0,
\end{equation}
where the two maps are respectively the first projection and the multiplication map. Here the group ${\rm GL}_2$ acts on $C'_g\simeq\bA^{2g+3}_{\rm sm}\times\PP^1$ by the formula
\[(f(x,y),p) \longmapsto (\det(M)^g f(M^{-1}(x,y)),M\cdot p),\]
where $M$ is an invertible matrix of rank $2$.

Again by Proposition \ref{prop:generic}, it is enough to check that the equality (\ref{eq:invariance}) holds when the invariant is evaluated at the generic point: if $\Spec(F)$ is the generic point of $C'_g\times{\rm GL}_2$, this is equivalent to say that 
\[(f\cdot \alpha_{g+1})(\Spec(F))= (f\cdot\alpha_{g+1}) (A\cdot\Spec(F)),\]
where $A$ is the matrix corresponding to $\Spec(F)\to {\rm GL}_2$ and $A\cdot\Spec(F)$ is regarded as an $F$-point of $C'_g$.

We already know that $\alpha_{g+1}$ descends to an invariant of $\Ccal_g$, hence $A$ must act trivially on it.
Moreover, the generic matrix $A$ sends the equivalence class $[f(p)] \in F^*/(F^*)^2$ to the class $[{\rm det}(A)^{g}f\circ A^{-1}(A(p))]$. The determinant is raised to an even power, thus $[{\rm det}(A)^{g}]=1$ in  $F^*/(F^*)^2$. We also have
\[ [f\circ A^{-1}(A(p))] = [f(A^{-1}A\cdot p)] = [f(p)] \] in $F^*/(F^*)^2$. This shows that ${\rm Pr}_1^*(f\cdot\alpha_{g+1})$ and ${\rm m}^*(f\cdot\alpha_{g+1})$ agree on the generic point, hence they are equal, concluding our proof.
\end{proof}

We still have to prove a rather relevant point: that the invariant we have created is not zero. For this, consider the open subset $U_0 \subset \bA^{2g+3}_{\rm sm}$ we defined earlier. The coefficient $x_0$ of a form is equal, up to squares, to its value at infinity, so taking the copy of $U_0$ inside $U'$ given by $U_0 \times \infty$ and pulling back the invariant $f \cdot \alpha_{g+1}$ we just defined (evaluating $f$ at infinity, in a sense) we obtain precisely the invariant $\tau \cdot \alpha_{g+1}$ from Proposition \ref{split}. We have just proven:

\begin{prop}
The element $f \cdot \alpha_{g+1}$ restricts to the invariant $\beta_{g+2}$ on $\Hcal_g$. In particular, it is non-zero and $\Het_2(k)$-linearly independent from the invariants coming from $\Brm {\rm S}_n$. 
\end{prop}
\begin{proof}
Consider the factorization
\[U_0 \xrightarrow{i} \mathcal{C}_g \xrightarrow{p} \Hcal_g.
\]
The pullback $p^*$ is an isomorphism as $\mathcal{C}_g \to \Hcal_g$ is a projective bundle. The pullback $(i\cdot p)^*$ is injective as $U_0 \to \Hcal_g$ is the composition of an open immersion and a smooth-Nisnevich morphism. By the observations above, the restrictions of $\beta_{g+2}$ and $f \cdot \alpha_{g+1}$ coincide on $U_0$, which implies that they are equal as elements of $\Inv(\Hcal_g)$.
\end{proof}

\begin{rmk}\label{rmk:not global}
	It is easy to see that every non-zero element $\tau \cdot \xi$ of $\tau \cdot A_T^0(\PP^{2g+2}_{\rm sm})$ which is not a multiple of $\beta_{g+2}$, regarded as an invariant of $U_0$, cannot be extended to a global invariant.
	
	Indeed, the generic point of $V_0$ defines the \'etale algebra $E_{gen}\times k$, where $E_{gen}$ is the generic \'etale algebra of degree $2g+1$. The boundary of $\tau \cdot \xi$ is equal to an invariant of $\Brm {\rm S}_{2g+1}$, whose value on $E_{gen}\times k$  is zero if and only if $\xi=0$.
\end{rmk}

\section{Multiplicative structure of $\Inv(\Hcal_g)$}\label{sec:mult}

In this Section we put together the results of the previous sections so to give a complete description of the multiplicative structure of $\Inv(\Hcal_g)$ for $g\geq 2$ even.

Recall that $\alpha_i$ denotes the degree $i$ cohomological invariant obtained by pulling back the $i^{\rm th}$ Stiefel-Whitney invariant along the morphism of stacks
\[ \Hcal_g \longrightarrow \Brm {\rm S}_{2g+2} \longrightarrow \Brm O_{2g+2} \]
Recall also that in Proposition \ref{split} we introduced a cohomological invariant $\beta_{g+2}$ of $\bA^{2g+3}_{\rm sm}$ which descends to a cohomological invariant of $\Hcal_g$.
\begin{thm}\label{thm:mult}
	Let $g\geq 2$ be an even number. Then the $\Het_2(k)$-module $\Inv(\Hcal_g)$ is freely generated by the invariants 
		 \[
		 1,\alpha_1,\alpha_2,\dots,\alpha_{g+1},\beta_{g+2}.
		 \]
		Define $\alpha_{g+2}=\lbrace 2 \rbrace \cdot \alpha_{g+1}, \, \alpha_{g+i}=0$ for $i > 2$. Then the ring structure of $\Inv(\Hcal_g)$ is determined by the following formulas:
	\begin{enumerate}
	\item	$\alpha_r\cdot\alpha_s = \lbrace -1 \rbrace^{m(r,s)} \cdot \alpha_{r+s-m(r,s)}$
	\item $\alpha_i\cdot \beta_{g+2} = 0 \,\, \text{for } \, i\neq g+1$ 
	\item $\alpha_{g+1} \cdot \beta_{g+2} = \lbrace -1 \rbrace^{g+1} \cdot \beta_{g+2}$ 
	\item $\beta_{g+2}\cdot \beta_{g+2} = \lbrace -1 \rbrace ^{g+2} \cdot \beta_{g+2}$
	\end{enumerate}
	where $m(r,s)$ is computed as follows: if we write $s=\sum_{i\in I} 2^i$ and $r=\sum_{j\in J} 2^J$, then $m(r,s)=\sum_{k\in I\cap J} 2^k$.
\end{thm}

\begin{proof}

Set $G:={\rm GL}_2$ and consider the $\Gm$-torsor 
\[ 
\Hcal_g = \left[ \bA^{2g+3}_{\rm sm} / G \right] \to \left[ \PP^{2g+2}_{\rm sm}/G \right].
\]
We complete it to a line bundle $\mathcal{L} \to \left[ \PP^{2g+2}_{\rm sm}/G \right]$. Then the zero section gives us a closed immersion $ \left[ \PP^{2g+2}_{\rm sm}/G \right] \xrightarrow{i} \mathcal{L}$ with complement isomorphic to $\left[\bA^{2g+3}_{\rm sm} / G \right] \xrightarrow{j} \mathcal{L}$. Identifying the equivariant Chow groups with coefficients of $\mathcal{L}$ with those of $\left[ \PP^{2g+2}_{\rm sm} / G \right]$, and consequently the inclusion of the zero section of $\mathcal{L}$ with the first Chern class of $\mathcal{L}$ (see \cite{PirCohHypEven}*{Def 2.2}), we get the localization exact sequence
\[
0 \to A_G^0(\PP^{2g+2}_{\rm sm}) \xrightarrow{j^*} A_G^0(\bA^{2g+3}_{\rm sm}) \xrightarrow{\partial} A_G^0(\PP^{2g+2}_{\rm sm})\xrightarrow{c_1(\mathcal{L})} A^1_G(\PP^{2g+2}_{\rm sm}).
\] 
We want to understand the kernel of the first Chern class $c_1(\mathcal{L})$. First, note that by construction $\partial(\beta_{g+2})=\alpha_{g+1}$, so we already know that $c_1(\mathcal{L})(\alpha_{g+1})=0$. We want to prove that there is no nonzero $\Het_2(k)$-linear combination of $1,\alpha_1,\ldots,\alpha_g$ that is annihilated by $c_1(\mathcal{L})$.

Let $\alpha=a_0 + a_1 \alpha_1 + \ldots + a_g \alpha_g$ be such a combination. Then the pullback of $c_1(\mathcal{L})\alpha$ to the non-equivariant group $A^1(\PP^{2g+2}_{\rm sm})$ must be zero. When restricted to $\PP^{2g+2}$ the line bundle $\mathcal{L}$ is $\mathcal{O}(-1)$, so its first Chern class is the hyperplane section $h$.

Consider the boundary map $\partial: A^1(\PP^{2g+2}_{\rm sm}) \to A^1(\Delta_{1,2g+2})$. We have $\partial(c_1(\mathcal{L})\alpha)=c_1(\mathcal{L}_{\mid \Delta_{1,2g+2}})\partial(\alpha)$ by the compatibility of Chern classes with the boundary maps. Now pull everything back to $A^{*}(\Delta_{1,2g+2} \setminus \Delta_{2,2g+2})\simeq A^{*}(\PP^{2g}_{\rm sm} \times \PP^1)$. Following the proof of Corollary \ref{cor:inv Pn}, we get that $\partial(\alpha_i)=\alpha_{i-1} + b_1 \alpha_{i-2} + \ldots + b_{i-1}$, with $b_1, \ldots, b_{i-1} \in \Het_2(k)$. This shows that inside $A^0(\PP^{2g}_{\rm sm} \times \PP^1)$ we have 
\[
\partial \alpha = a_g \alpha_{g-1} + \beta,\quad \beta \in \langle 1, \ldots, \alpha_{g-2} \rangle.
\]
Moreover, the pullback of $\mathcal{O}(-1)$ to $\PP^{2g}\times \PP^1$ is equal to $\mathcal{O}_{\PP^{2g}}(-1)\otimes \mathcal{O}_{\PP^1}(-1)^2$, so modulo $2$ its Chern class is the same as the Chern class of $\mathcal{O}_{\PP^{2g}}(-1)$. In other words, we have an equation
\[ h \cdot (a_g \alpha_{g-1} + \beta)= 0 \]
inside $A^1(\PP_{\rm sm}^{2g})$. We can repeat these exact same steps $g-1$ more times to obtain $h \cdot a_g = 0$ inside $A^0(\PP^2_{\rm sm})$. Now, we proved in Corollary \ref{cor:inv Pn} that the image of $A^0(\Delta_{1,n})$ in $A^1(\PP^n)$ is trivial. Then $A^1(\PP^n)=h \cdot \Het_2(k)$ by the projective bundle formula and $h \cdot a_g=0$ implies $a_g=0$. We can plug this result in our initial formula and repeat the reasoning for the remaining coefficients: each step is exactly the same. At the end, we conclude that all coefficients must be zero, as claimed.

To obtain a description of the multiplicative structure, let $\xi$ be the generic point of some equivariant approximation $(\PP^{2g+2}_{\rm sm}\times U)/G$ of $\left[\PP_{\rm sm}^{2g+2}/G\right]$, and let $K$ be the corresponding field. Then the generic point $\xi'$ of the corresponding equivariant approximation $(\bA^{2g+3}_{\rm sm}\times U)/G$ of $\Hcal_g$ has corresponding field $K(t)$, being the generic point of a $\Gm$-torsor. 
	
The cohomological invariants of $\left[\PP_{\rm sm}^{2g+2}/G\right]$ inject into $\Het_2(K)$, and their product formulas are the same as the ones in $\Inv(\Brm {\rm S}_n)$. The invariants of $\Hcal_g$ inject into $\Het_2(K(t))$, and by construction they are freely generated by $\alpha_1, \ldots, \alpha_{g+1}, \beta_{g+2}=\alpha_{g+1} \cdot t$. We have the relation $t \cdot t = \lbrace -1 \rbrace \cdot t$ coming from Milnor's $\rm K$-theory. These relations are sufficient to completely describe the ring $\Inv(\Hcal_g)$.

In particular, the relation $\alpha_i\cdot\beta_{g+2}=0$ for $i\neq g+1$ comes from the fact that $\beta_{g+2}=\alpha_{g+1}\cdot t$ in $\Het_2(K(t))$, and $\alpha_{i}\cdot\alpha_{g+1}=0$ for $i\neq g+1$. Indeed, write $i=\sum_{j\in R} 2^j$, $g+1=\sum_{j\in S} 2^j$ and $m=\sum_{j\in R\cap S} 2^j$. Then for $i\neq g+1$ we have $m<i$ and from the formula of \cite{GMS}*{Remark 17.4} we know that
\[ \alpha_i\cdot\alpha_{g+1}=\{-1\}^m\alpha_{g+1+i-m}=0, \]
because $\alpha_k=0$ for $k>g+1$.
\end{proof}
	
\section{The $g$ odd case}\label{sec:g odd}

We use the last Section for an informal discussion of the case of odd genus. There are two (related) main differences. First, differently from ${\rm GL}_2$, the projective linear group ${\rm PGL}_2$ is not special, in the sense that every torsor is Zariski trivial. Consequently, it is possible for $\Brm {\rm PGL}_2$ to have nontrivial cohomological invariants, and in fact $\Inv(\Brm {\rm PGL}_2)$ is freely generated as $\Het_2(k)$-module by $1$ and an invariant $w_2$ of degree two, which will appear in our computations. The second difference is that, given a hyperelliptic curve $C$, the quotient $C'=C/\iota$ does not need to be a trivial conic. In fact, if we see $\Brm {\rm PGL}_2$ as the classifying stack of conic bundles, the map $\Hcal_g \to \Brm {\rm PGL}_2$ sends a hyperelliptic curve $C$ with involution $\iota$ to the class of $C'=C/\iota$.

Nonetheless, Corollary \ref{cor:inv Pn} tells us the following:

\begin{prop}\label{InvSnOdd}
Let $g\geq 3$ be an odd integer. The ring of cohomological invariants of $\left[\PP^{2g+2}_{sm}/{\rm PGL}_2\times \mathbb{G}_m\right]$ is freely generated as $\Het_2(k)$-module by 
\[1,\alpha_1,w_2,\alpha_2,\alpha_3,\dots,\alpha_{g+1}, \] where:
\begin{itemize}
	\item The $\alpha_i$ are the cohomological invariants pulled back along the morphism  
	\[
	\left[\PP^{2g+2}_{sm}/{\rm PGL}_2\times \mathbb{G}_m\right]\to \Brm {\rm S}_{2g+2} \to \Brm O_{2g+2}. 
	\]
	\item The cohomological invariant $w_2$ is pulled back along the morphism 
	\[
	\left[\PP^{2g+2}_{sm}/{\rm PGL}_2\times \Gm\right]\to \Brm {\rm PGL}_2.
	\]
\end{itemize}
Moreover, we have an exact sequence
\[
0 \to \Inv(\left[ \PP^{2g+2}/{\rm PGL}_2 \times \Gm\right]) \to \Inv(\Hcal_g) \to \Het_2(k)
\]
Where the last map lowers degree by $g+2$.
\end{prop}
\begin{proof}
The first two statements are a direct consequence of Corollary \ref{cor:inv Pn}. Now, set $G={\rm PGL}_2 \times \Gm$. As in the proof of Theorem \ref{thm:mult}, the $\Gm$-torsor $\Hcal_g \to \left[ \PP^{2g+2}/{\rm PGL}_2 \times \Gm\right]$ induces an exact sequence
\[
0 \to A^0_G(\PP^{2g+2}_{\rm sm}) \to A^0(\bA^{2g+3}_{\rm sm}) \xrightarrow{\partial} A^0_G(\PP^{2g+2}_{\rm sm}) \xrightarrow{c_1(\mathcal{L})} A^1_G(\PP^{2g+2}_{\rm sm}).
\]
In this case, by \cite{PirCohHypThree}*{Lemma 25}, we have $\mathcal{L}=\mathcal{O}_{\PP^{2g+2}}(-1) \otimes \mathcal{E}^2$, where $\mathcal{E}$ comes from the canonical representation of $\Gm$. In particular the class of $c_1(\mathcal{L})$ is equivalent modulo $2$ to that of $\mathcal{O}_{\PP^{2g+2}}(-1)$, which we will again denote by $h$.

We have a formula
\[
 A^{*}_{{\rm PGL}_2 \times \Gm}({\rm Spec}(k))=A^{*}_{{\rm PGL}_2 }({\rm Spec}(k))\left[ s \right]
\]
where $s=c_1(\mathcal{E})$ and the square bracket denotes a polynomial ring. This is proven in \cite{PirCohHypThree}*{Prop. 14}. It can also be easily seen reasoning exactly as in the proof of the last point of Lemma \ref{lm:GL2PGL2}.

By the projective bundle formula applied to $\left[ \PP^{2g+2} /G \right] \to \Brm G$, we obtain 
\[
A^1_G(\PP^{2g+2}) = h \cdot A^0_{{\rm PGL}_2}({\rm Spec}(k)) \oplus s \cdot A^0_{{\rm PGL}_2}({\rm Spec}(k)) \oplus A^1_{{\rm PGL}_2}({\rm Spec}(k)). 
\]

As the image of $i_*A^0(\Delta_{1}^n)$ is zero by Corollary \ref{cor:inv Pn}, we conclude that $h \cdot b w_2$ is non-zero for any $b \in \Het_2(k)$.

Assume a combination $a_0 + a_1 \alpha_1 + \ldots + a_{g} \alpha_g + b w_2$ is annihilated by $h$. Then we can pull back to the non-equivariant ring $A^{*}(\PP^{2g+2}_{\rm sm})$. The class $w_2$ goes to zero, while the map is injective on the elements coming from $\Inv(\Brm {\rm S}_{2g+2})$, so we obtain the equation
\[ h\cdot (a_0 + a_1 \alpha_1 + \ldots + a_{g} \alpha_g) = 0 \in A^{1}(\PP_{\rm sm}^{2g+2}) \]
but we already saw in the proof of Theorem \ref{thm:mult} that this implies that all the $a_i$, and consequently $b$, must be zero.

Finally, consider a general element 
\[
\zeta = a_0 + a_1 \alpha_1 + \ldots + a_{g+1} \alpha_{g+1} + b w_2.
\]
As the coefficients $a_i$ are uniquely determined, we can consider the projection ${\rm Ker}(c_1(\mathcal{L}) \to \Het_2(k))$ given by $\zeta \xrightarrow{\pi} a_{g+1}$. We claim that this map is injective. Indeed, assume $\zeta, \zeta' \in {\rm Ker}(c_1(\mathcal{L}))$ have the same image. Then $\zeta-\zeta' \in {\rm Ker}(c_1(\mathcal{L}))$ is in the form $a_0 + a_1 \alpha_1 + \ldots + a_{g} \alpha_g + b w_2$, so it must be zero. This shows that we have an exact sequence
\[
0 \to \Inv(\left[ \PP^{2g+2}/{\rm PGL}_2 \times \Gm\right]) \to \Inv(\Hcal_g) \xrightarrow{\pi \circ \partial} \Het_2(k)
\]
concluding the proof.
\end{proof}

Unfortunately, both approaches used in Section \ref{sec:g even} to construct the last invariant break down in the odd case, due to two reasons:

\begin{itemize}
\item The quotient map $\bA^{2g+3}_{\rm sm}\rightarrow\left[\PP^{2g+2}_{\rm sm}/{\rm PGL}_2\times \mathbb{G}_m\right]$ is not smooth-Nisnevich, so we cannot glue invariants along it.
\item The universal conic $\mathcal{C} \to \Hcal_g$ is not the projectivization of a vector bundle, and thus it does not necessarily induce an isomorphism on cohomological invariants.
\end{itemize}

We can circumvent the first problem by using a different presentation for $\Hcal_g$, which the first author used in \cite{DilChowHyp, DilCohHypOdd} to compute the Chow ring and the cohomological invariants of $\Hcal_g$ when $g$ is odd.

Let $\bA(2,m)$ be the scheme representing $m$-forms in three variables, and let $\bA(2,m)_{\rm sm}$ be the open subset of smooth forms. Let $V_n$ be the vector bundle over $\bA(2,2)\sm$ defined as the quotient of the trivial vector bundle $\bA(2,2)\sm\times\bA(2,n)$ by the trivial vector subbundle $\bA(2,2)\sm\times\bA(2,n-2)$, where the inclusion of the latter vector bundle in the first one sends a pair $(q,f)$ to $(q,qf)$.

\begin{thm}[\cite{DilChowHyp}*{Theorem 2.9}]
Let $g\geq 3$ be odd, and let $n=g+1$. There is an open subset $U_{n} \subset V_{n}$ and an action of $\GLt \times \mathbb{G}_m$ on $U_{n}$ such that $\left[ U_{n} / \GLt \times \mathbb{G}_m\right] = \Hcal_g$.
\end{thm}

In this setting, the $\mathbb{G}_m$-torsor $\bA^{2g+3}_{sm} \rightarrow \PP^{2g+3}$ corresponds to the $\mathbb{G}_m$-torsor $U_n \to \PP(U_{n})$, i.e. the quotients $\left[ \PP(U_{n})/\GLt \times \mathbb{G}_m \right]$ and  $\left[\PP^{2g+2}_{sm}/{\rm PGL}_2\times \mathbb{G}_m\right]$ are naturally isomorphic.

Differently from ${\rm PGL}_2$, $\GLt$ is a special group, which means that the cohomological invariants of $\Hcal_g$ will inject into the cohomological invariants of $U_{n}$, and the cohomological invariants of $\left[ \PP(U_{n})/\GLt \times \mathbb{G}_m\right]$ will inject into the cohomological invariants of $\PP(U_{n})$. All we have to do is understand whether the $\mathbb{G}_m$-torsor 
\[
\Hcal_g = \left[ U_{n}/\GLt \times \mathbb{G}_m\right] \to \left[ \PP(U_{n})/\GLt \times \mathbb{G}_m\right]
\]
creates any new invariant. Following what we did in the even case, we can try to trivialize the projective bundle and understand the ramification maps that come out.

We do the following. Let a form in $\bA(2,2)_{\rm sm}$ be denoted by $q(x_0,x_1,x_2)$, and let $\bA(2,2)^0_{\rm sm}$ be the open subset of $\bA(2,2)_{\rm sm}$ where the coefficient of $x_0^2$ is not zero. Over $\bA(2,2)^0_{\rm sm}$, given a class $(q, \left[ f \right])$, there is a unique representative $(q,f)$ such that no term divisible by $x_0^2$ appears in $f$. This gives a trivialization of $V_n$ over $\bA(2,2)^0_{\rm sm}$. In other words, write $V_n^0=V_n \times_{\bA(2,2)_{\rm sm}}\bA(2,2)^0_{\rm sm}$ . Then $V_n^0=\bA^{2n-1} \times \bA(2,2)^0_{\rm sm}$ and the $\mathbb{G}_m$ torsor is just \[\bA^{2n-1}\times \bA(2,2)^0_{\rm sm} \rightarrow \PP^{2n-2}\times \bA(2,2)^0_{\rm sm} .\]
We can trivialize this torsor by requiring an additional coefficient to be nonzero, say the highest power of $x_{2}$. Note that this only makes sense on $V_n^0$. Denote this new open subset by $V_n^{0,2}$, and its intersection with $U_n$ by $U_n^{0,2}$. 

Reasoning as in the beginning of the proof of Proposition \ref{split} we have the inclusion 
\[\Inv(\Hcal_g) \subset \Inv( U^{0,2}_{n}) = 
\Inv(\PP(U^{0,2}_{n})) \oplus \alpha \cdot \Inv(\PP(U^{0,2}_{n}))\]
where $\alpha$ can be locally described as the equation of the complement of $U^{0,2}_n$ in $U_n$. Unfortunately in this case being in the complement of $U_n^{0,2}$ does not appear to yield clear information on the curve or its Weierstrass divisor, so we are unable to proceed further. All we can do is use the knowledge of the additive structure of the invariants over an algebraically closed field \cite[Thm. 2.1]{DilCohHypOdd} to get some \emph{a posteriori} understanding of the multiplicative structure. 

\begin{prop}\label{thm:multOdd}
	Let $g\geq 2$ be an odd number, and assume $k$ is algebraically closed. Then the $\Het_2(k)$-module $\Inv(\Hcal_g)$ is freely generated by the invariants \[1,\alpha_1,w_2, \alpha_2,\dots,\alpha_{g+1},\beta_{g+2}.\] Now, define $\alpha_i = 0$ for $i > g+1$. Then we have the formulas:
	\begin{enumerate}
		\item $\alpha_r\cdot\alpha_s =  \alpha_{r+s} \,\,\, \text{if} \,\,\, m(r,s)=0$, and zero otherwise.
		\item $\alpha_i\cdot \beta_{g+2} = 0$ .
		\item $\beta_{g+2}\cdot \beta_{g+2} = 0$
		\item $w_2 \cdot \alpha_{i} = w_2 \cdot \beta_{g+2} = 0$.
		\end{enumerate}
		where $m(r,s)$ is computed as follows: if we write $s=\sum_{i\in I} 2^i$ and $r=\sum_{j\in J} 2^J$, then $m(r,s)=\sum_{k\in I\cap J} 2^k$.
\end{prop}
\begin{proof}
The additive description of $\Inv(\Hcal_g)$ is \cite{DilCohHypOdd}*{Thm. 2.1}. The multiplication formulas between the $\alpha_i$ and $\beta_{g+2}$ are immediate from the formulas in $\Inv(\Brm {\rm S}_{2g+2})$ and the fact that $\lbrace -1 \rbrace=0$ when $k$ contains a square root of $-1$. 

A multiple of $w_2$ has to pull back to zero in $\Inv(\bA^{2g+3}_{\rm sm})$, but the kernel of the pullback is exactly $w_2$ itself.
\end{proof}

Without assuming that the base field is algebraically closed we do not get any information on the products of $w_2$ with other elements of the ring. We close with two open questions which require further investigation:

\begin{itemize}
\item What does being in the complement of $U_n^{0,2}$ tell us about the Weierstrass divisor of a curve?
\item Can we understand the products $w_2 \cdot \alpha_i$ when $k$ is not algebraically closed?
\end{itemize}

\appendix
\section{Cohomological invariants of $\overline{\Hcal}_g$}\label{app:InvHbar}
Let $\overline{\Hcal}_g$ denotes the moduli stack of stable hyperelliptic curves, i.e. the closure of $\Hcal_g$ inside the moduli stack $\overline{\Mcal}_g$ of stable curves of genus $g$. It is well known that $\overline{\Hcal}_g$ is a smooth Deligne--Mumford stack.

Let us briefly recall what are the irreducible components of $\partial \overline{\Hcal}_g:=\overline{\Hcal}_g\smallsetminus\Hcal_g$. We have:
\[ \partial \overline{\Hcal}_g = \Xi_{\rm irr} + \Xi_1 + \ldots + \Xi_{\lfloor \frac{g-1}{2} \rfloor} + \Delta_1 + \ldots + \Delta_{\lfloor \frac{g}{2} \rfloor} \] 
The generic points of the irreducible components above can be described as follows (for a more detailed description see \cite{CH}):
recall that any stable hyperelliptic curve can be realized as a double cover $C\to \Gamma$ over a chain of curves of genus $0$. The generic point of $\Xi_{\rm irr}$ is given by a singular curve which admits a $2:1$ morphism onto an irreducible rational curve $\Gamma$; the generic point of $\Xi_i$ consists of a singular curve with two irreducible components $C=C_1\cup C_2$ of genus $i$ and $g-i-1$, which is a double cover of a pair of lines $\Gamma=\Gamma_1 \cup \Gamma_2$, and such that the fibre over a node is made of two distinct points; the generic point of $\Delta_i$ is a singular curve similar to the generic point of $\Xi_i$, albeit the two irreducible components have genus $i$ and $g-i$, and there is only one point over the node of $\Gamma$.
%
%

As by Proposition \ref{prop:generic} an open immersion induces an injective pullback on cohomological invariants, the cohomological invariants of $\overline{\Hcal}_g$ inject into those of $\Hcal_g$. We want to show that they are trivial for all $g$.

\begin{thm}\label{thm:Inv Hgbar}
	Let $k$ be a field of characteristic $\neq 2$, and let $g\geq 2$. Then:
	\[ \Inv(\overline{\Hcal}_g, {\rm H}_{\ell}) \simeq \Het_{\ell}(k). \]
\end{thm}

Before we prove it, let us remark that Theorem \ref{thm:Inv Hgbar} should not come as a surprise: indeed, we showed that almost all the cohomological invariants of smooth hyperelliptic curves come from the morphism $\Hcal_g\to\Brm {\rm S}_{2g+2}$ induced by the Weierstrass divisor, which is \'{e}tale on the base. On the other hand, the ramification divisor of stable hyperelliptic curves is not \'{e}tale anymore, hence there is not a natural extension of the morphism $\Hcal_g\to\Brm {\rm S}_{2g+2}$ to $\overline{\Hcal}_g$.
	
For the same reason, when $g$ is odd we cannot expect that the invariant coming from $\Brm {\PGLt}$ extends to $\overline{\Hcal}_g$: the natural quotient by the hyperelliptic involution of a stable hyperelliptic curve with two or more components is a chain of rational curves, and it is easy to show that $w_2$ cannot be extended to these singular genus $0$ curves.

\begin{rmk}\label{rmk: admissible}
The compactification $\widetilde{\Hcal}_g$ by admissible coverings should behave much better from the point of view of cohomological invariants: there should exist an extended morphism $\widetilde{\Hcal}_g\to\Brm {\rm S}_{2g+2}$ which would allow us to extend the invariants $\alpha_1,\dots,\alpha_{g+1}$ from $\Hcal_g$ to $\widetilde{\Hcal}_g$. This will be the subject of a subsequent paper.

Note that while $\overline{\Hcal}_g$ and $\widetilde{\Hcal}_g$ share the same moduli space, the automorphism groups of the objects they parametrize are not the same in general, hence the stacks cannot be equal either. 

To see this, pick an algebraically closed field $K$ and consider a curve $X/K$ with two components, namely a smooth hyperelliptic curve $C$ of genus $g-1$ and a rational curve $E=\PP^1$ meeting the other component in two distinct points $p_1=(0:1)$ and $p_2=(1:0)$. Assume that the involution $\iota$ of $C$ exchanges $p_1$ and $p_2$. Then there is an admissible involution $\iota_X$ on $X$ acting as $\iota$ on $C$ and as $(x:y)\mapsto (y:x)$ on $E$, and $X\to X'=X/\iota_X$ is an admissible double covering of genus $g$, according to the definition given in \cite{HM}*{Section 4}. Its stable model $X^{\rm st}$ is the stable hyperelliptic curve obtained by contracting the rational component $E$.

By definition, the automorphism group of $X\to X'$ is generated by the automorphisms of $X$ that commute with the admissible involution. For an admissible double covering $X$ as above, when $C$ is generic, the automorphism group of its stable model $X^{\rm st}$ is generated by the hyperelliptic involution. On the other hand, the admissible covering has an additional non-trivial automorphism of order $2$, which acts as the identity on $C$ and as $(x:y)\mapsto (-x:y)$ on $E$. Therefore, the stacks $\overline{\Hcal}_g$ and $\widetilde{\Hcal}_g$ cannot be isomorphic.
\end{rmk}

\begin{lm}\label{lm:Hg prime}
Consider the open substack $\overline{\Hcal}_g'$ of $\overline{\Hcal}_g$ defined as:
\[ \overline{\Hcal}_g' := \overline{\Hcal}_g \smallsetminus \left( \cup_i \Xi_i \cup_j \Delta_j \right). \]
Then we have:
\[ \overline{\Hcal}_g'\simeq \left[ U_g/G \right] \]
where $G$ is ${\rm GL}_2$ when $g$ is even and ${\rm PGL}_2\times\Gm$ when $g$ is odd, the action is the same as in Theorem \ref{thm:presentation}, and $U_g$ is an open subset of $\bA^{2g+3}$ whose complement has codimension greater than $1$.
\end{lm}
\begin{proof}
Observe that a curve in $\overline{\Hcal}_{g}'$ is a uniform cyclic covering of degree $2$ of the projective line, whose ramification divisor is a $0$-dimensional scheme of length $2g+2$ having at most double points: if $\mathcal{H}(1,2,2g+2)$ denotes the stack of uniform cyclic double coverings of the projective line with ramification divisor of length $2g+2$, then by \cite[Thm. 4.1]{ArsVis} we have 
\[\overline{\Hcal}_g'\hookrightarrow \mathcal{H}(1,2,2g+2) \simeq \left[ \mathbb{A}^{2g+3}\smallsetminus\{0\}/ G \right].\]
By construction, the image of $\overline{\Hcal}_g'$ in $\left[ \mathbb{A}^{2g+3}\smallsetminus\{0\}/ G \right]$ is the open substack $[U_g/G]$, where $U_g$ is the open subscheme of forms having at most double roots. The complement of $U_g$ in $\bA^{2g+3}$ has codimension $\geq 2$, thus concluding the proof.


\end{proof}

\begin{proof}[Proof of Theorem \ref{thm:Inv Hgbar}]

By Proposition \ref{prop:generic}, the pullback through the open immersion gives us an injection $\Inv(\overline{\Hcal}_g)\hookrightarrow\Inv(\overline{\Hcal}_g')$.

By Lemma \ref{lm:Hg prime} we know that $\overline{\Hcal}_g'\simeq \left[ U_g/G \right]$, with $U_g\subset\bA^{2g+3}$. As the complement of $U_g$ in $\bA^{2g+3}$ has codimension $>1$ we can apply Proposition \ref{Prop:homot inv} to obtain the following:
\[  A^0_G(U_g,{\rm H}_{\ell})\simeq A^0_G(\bA^{2g+3},{\rm H}_{\ell}) \simeq A^0_G(\Spec(k),{\rm H}_{\ell}). \]
As already observed before, being ${\rm GL}_2$ special implies that the universal ${\rm GL}_2$-torsor $\Spec(k)\to\Brm {\rm GL}_2$ is smmooth-Nisnevich, hence the pullback of cohomological invariants along this map is injective (Theorem \ref{thm:sheaf}). Therefore, we deduce that $ {\rm Inv}^{\bullet}(\overline{\Hcal}_g',{\rm H}_{\ell})\simeq \Het_{\ell}(k)$ for $g$ even.

For the $g$ odd case, the reasoning above and Lemma \ref{lm:GL2PGL2} give us the result immediately for odd $\ell$, and for $\ell=2$ we get:
\[ \Inv(\overline{\Hcal}_g) \subseteq \Inv(\overline{\Hcal}_g')\subseteq \Inv({\rm B} {\rm PGL}_2 \times \Gm) = \Het_2(k)\oplus \Het_2(k) \cdot w_2 \]
where $w_2$ is the cohomological invariant coming from $\Brm\PGLt$. We only have to show that this invariant does not extend to the whole $\overline{\Hcal}_g$.

 We write $\bA(2,2)$ for the affine space of homogeneous forms of degree $2$ in three variables, we write $\bA(2,2)_r$ for the locally closed subscheme of forms of rank $r$, and $\bA(2,2)_{\left[3,2\right]}$ for the open subscheme of forms of rank at least $2$. By \cite{DilCohHypOdd}*{Sec. 3} there is a natural ${\rm GL}_3$ action on $\bA(2,2)$ which fixes the subschemes we defined, namely the one defined by the formula
\[ A\cdot f(x,y,z) := {\rm det}(A)f(A^{-1}(x,y,z)).\]
The quotient $\left[ \bA(2,2)/{\rm GL}_3 \right]$ is the moduli stack of conic bundles. The open subset $\left[ \bA(2,2)_{3} / \GLt \right]$ parametrizes smooth conics, and consequently is isomorphic to ${\rm BPGL}_2$. The open substack $\left[ \bA(2,2)_{\left[3,2\right]} / \GLt \right]$ parametrizes conics of rank $\geq 2$.

Define $\overline{\Hcal}_g^{[0,1]}$ as the stack of stable hyperelliptic curves of genus $g$ having at most one node, and set
\[ \overline{\Hcal}_g^0 := \overline{\Hcal}_g^{[0,1]}\cap \overline{\Hcal}_g', \quad \overline{\Hcal}_g^1:= \overline{\Hcal}_g^{[0,1]} \smallsetminus \overline{\Hcal}_g^0,\]
where $\overline{\Hcal}_g'$ is the stack defined in Lemma \ref{lm:Hg prime}.
Given a stable hyperelliptic curve, we can consider its quotient by the hyperelliptic involution, which we denote $\Gamma$: in general, $\Gamma$ will be a chain of rational curves. If we restrict to the points in $\overline{\Hcal}_g^{[0,1]}$, then $\Gamma$ will be a curve of genus $0$ having at most one node.

Therefore, we can define a morphism of stacks
\[f: \overline{\Hcal}_g^{[0,1]} \longrightarrow \left[\bA(2,2)_{[3,2]}/\GLt \right]\]
by sending a hyperelliptic curve to its quotient by the hyperelliptic involution. Observe that the preimage of $\Brm {\rm PGL}_2\simeq \left[\bA(2,2)_3/\GLt \right]$ is equal to $\overline{\Hcal}_g^0$, and the preimage of $\left[ \bA(2,2)_2/\GLt \right]$ is $\overline{\Hcal}_g^1$. In other terms, we have a cartesian diagram
\[ \xymatrix{\overline{\Hcal}_g^1 \ar[r] \ar[d] & \overline{\Hcal}_g^{[0,1]} \ar[d] \\ \left[ \bA(2,2)_2/\GLt \right] \ar[r] & \left[ \bA(2,2)_{[3,2]}/\GLt\right]}. \]
Combining this diagram with the localization exact sequence we get the following commutative diagram of equivariant Chow groups with coefficients
\[ \xymatrix{  A^0(\left[ \bA(2,2)_{3}/\GLt\right]) \ar[r]^{\partial} \ar[d]^{f^*} & A^0( \left[ \bA(2,2)_{2}/\GLt\right]) \ar[d]^{f^*}  \\ A^0(\overline{\Hcal}_g^{0}) \ar[r]^{\partial} & A^0(\overline{\Hcal}_g^{1}) }  \]
where the groups on the bottom row are well defined because both stacks are quotient stacks.

The vertical arrow on the right is injective: the map $f$ is representable and smooth, and given a $K$-point of $\left[ \bA(2,2)_{2}/\GLt\right]$, i.e. a ternary quadric $\Gamma$ of rank $2$ defined over some field $K$, we can always construct a stable hyperelliptic curve whose quotient by the hyperelliptic involution is exactly $\Gamma$ or, in other terms, we can always lift a map $\Spec(K)\to \left[ \bA(2,2)_{2}/\GLt\right]$ to a map to $\overline{\Hcal}_g^{1}$. Henceforth, the morphism $f$ is smooth-Nisnevich.

Moreover, consider the localization exact sequence
\[
0 \to A^0_{\GLt}(\bA(2,2)_{\left[3,2\right]} \to A^0_{\GLt}(\bA(2,2)_3) \xrightarrow{\partial} A^0_{\GLt}(\bA(2,2)_2). 
\]
By combining homotopy invariance (Proposition \ref{Prop:homot inv}) and the fact that $\GLt$ is special, we have that $A^0_{\GLt}(\bA(2,2))\simeq\Het_2(k)$. Consequently, again by Proposition \ref{Prop:homot inv} we have
\[ A^0\left(\left[ \bA(2,2)_{[3,2]}/\GLt\right]\right)\simeq \Het_2(k) \]
as the complement of $\bA(2,2)_{[3,2]}$ in $\bA(2,2)$ has codimension $>1$. This shows that the element $w_2 \in A^0_{\GLt}(\bA(2,2)_3)$ cannot map to zero in $A^0_{\GLt}(\bA(2,2)_2)$.

Putting together these last two observations, we conclude that $\partial w_2\neq 0$ in $A^0(\overline{\Hcal}_g^{1})$, hence $w_2$ does not extend to an invariant of $\overline{\Hcal}_g^{[0,1]}$, so it cannot extend to $\overline{\Hcal}_g \supset \overline{\Hcal}_g^{[0,1]}$. This concludes the proof.
\end{proof}

\section{Torsion in the Picard group of $\overline{\Hcal}_g$}\label{app:PicHbar}


Let $\lambda$ denote the class of the Hodge line bundle in the Picard group of $\overline{\Hcal}_g$. When the base field has characteristic zero, Cornalba and Harris proved in \cite{CH} that the following formula holds up to torsion:
\begin{equation}\label{eq:CH}
(8g+4)\lambda=g[\Xi_{irr}] + 2 \sum_{i=1}^{\lfloor \frac{g-1}{2} \rfloor} (i+1)(g-1)[\Xi_i] + 4 \sum_{j=1}^{\lfloor \frac{g}{2} \rfloor} j(g-j)[\Delta_j].
\end{equation}
The formula (\ref{eq:CH}) is usually called \emph{Cornalba-Harris equality}. In \cite{Yam} the Cornalba-Harris equality was shown to hold (up to torsion) also over fields of positive characteristic. 

Afterward, assuming that the base field is $k=\mathbb{C}$, Cornalba showed in \cite{Cor} that (\ref{eq:CH}) holds integrally, and he suggested that the formula should be true over any base field.  
The main result of this Appendix is the following.
\begin{thm}\label{thm:Pic}
Let $k$ be a field of characteristic $q\neq 2$. Then ${\rm Pic}(\overline{\Hcal}_g)$ is torsion free. In particular, the Cornalba-Harris equality (\ref{eq:CH}) holds integrally.
\end{thm}
When $g=2$, this result can be obtained as a special case of broader results. Larson, in \cite{Lars}, computed the Chow ring of $\overline{\mathcal{H}}_2$, and consequently its Picard group, over a field of characteristic not $2$ or $3$. Recently, Fringuelli and Viviani \cite{FriViv} computed the Picard group of the stack $\overline{\mathcal{M}}_{g,n}$ of stable $n$-pointed curves over any Noetherian base where $2$ is invertible, and over any field when $g \leq 5$. In particular, they compute ${\rm Pic}(\overline{\Hcal}_2)$ over any field. 

Recall that given an algebraic stack $\mathcal{S}$ we can see ${\rm H}^1_{\rm fppf}(\mathcal{S}, \mu_n)$ as the group of pairs $(\mathcal{L},\alpha)$ where $\mathcal{L}$ is an invertible sheaf and $\alpha:\mathcal{L}^{\otimes n} \rightarrow \OO_{\mathcal{S}}$ is an isomorphism. Morever, we have a Kummer exact sequence
\[
0 \rightarrow \Gamma(\mathcal{S},\mathcal{O}^*_{\mathcal{S}})/\Gamma(\mathcal{S},\mathcal{O}^*_{\mathcal{S}})^n \rightarrow {\rm H}^1_{\rm fppf}(\mathcal{S}, \mu_n) \rightarrow {\rm Pic}(\mathcal{S}) \xrightarrow{\cdot n} {\rm Pic}(\mathcal{S}). 
\]
Note that $k^*/(k^*)^n$ injects into $\Gamma(\mathcal{O}^*_{\overline{\Hcal}_g}(\overline{\Hcal}_g))/\Gamma(\mathcal{O}^*_{\overline{\Hcal}_g}(\overline{\Hcal}_g))^n$, because $\overline{\Hcal}_g$ has a rational point, and if there were a global $n$-th root of an element of $k^*$ defined over $\overline{\Hcal}_g$ then there would have to be one at every residue field. 
Then by the exact sequence above if we show that for all $n$ we have ${\rm H}^1_{\rm fppf}(\overline{\Hcal}_g, \mu_n)=k^*/(k^*)^n$ it will immediately imply that ${\rm Pic}(\overline{\Hcal}_g)$ has no torsion.

The approach will be similar to what we did for cohomological invariants; we will show that the statement holds for the open substack $\overline{\Hcal}_g '$ defined in Lemma \ref{lm:Hg prime}, and then that this implies the same for $\overline{\Hcal}_g$.



We begin by proving some simple invariance results for ${\rm H}^1_{\rm fppf}(-, \mu_n)$.

\begin{lm}\label{H1inv}
Let $\mathcal{Y} \xrightarrow{f} \mathcal{X}$ be a morphism of smooth algebraic stacks over $k$.
\begin{itemize}
\item If $f$ is a vector bundle then $f^*:{\rm H}^1_{\rm fppf}(\mathcal{X},\mu_n)\rightarrow {\rm H}^1_{\rm fppf}(\mathcal{Y},\mu_n)$ is an isomorphism.
\item If is an open immersion whose complement has codimension $\geq 2$ then $f^*:{\rm H}^1_{\rm fppf}(\mathcal{X},\mu_n)\rightarrow {\rm H}^1_{\rm fppf}(\mathcal{Y},\mu_n)$ is an isomorphism.
\end{itemize}
\end{lm}
\begin{proof}
The two points follow from the Kummer sequence, together with the fact that $f$ induces an isomorphism on Picard groups \cite{PTT}*{Prop. 1.9} and on global sections of $\OO^*$.
\end{proof}

The Lemma we just proved will allow us to reduce the computation of ${\rm H}^1_{\rm fppf}(\overline{\Hcal}_g',\mu_n)$ to the computation of ${\rm H}^1_{\rm fppf}(\Brm G,\mu_n)$ for $G={\rm GL}_2$ or $G={\rm PGL}_2 \times \Gm$. These groups can be determined by using equivariant approximations. The next Lemma relates the $\mu_n$-torsors over a scheme with those over an open subset.

\begin{lm}\label{H1inj}
Let $X$ be a smooth, irreducible and separated scheme. If $f:U\rightarrow X$ is an open immersion then $f^*:{\rm H}^1_{\rm fppf}(X,\mu_n)\rightarrow {\rm H}^1_{\rm fppf}(U,\mu_n)$ is injective.
\end{lm}
\begin{proof}

Assume that we have an element $(\mathcal{L},\alpha)$ whose restriction to $U$ is trivial. This amounts to the map $\alpha$ having an $n$-th root $\beta$ defined over $U$. We want to show that $\beta$ extends to $X$.

The question is local on $X$; given two extensions $\beta_{U'}, \beta_{U''}$ defined respectively on open subsets $U',U'' \supset U$, they agree on $U \cap U'$ as $X$ is reduced and separated and thus glue to an extension on $U' \cup U''$. Using this, we can reduce to the spectrum of the DVR $\mathcal{O}_{X,p}$ where $p$ is a closed point of codimension one. If we can extend $\beta$ to ${\rm Spec}(\mathcal{O}_{X,p})$, then we can extend it to some open subset of $X$ containing $U$ and $p$, and consequently to all points of codimension $1$. Now, an isomorphism of line bundles outside of a closed subset of codimension $\geq 2$ always extends to an isomorphism, so this proves our claim.

Passing to the spectrum of $\mathcal{O}_{X,p}$, the line bundle is automatically trivial, $\alpha \in \OO^*(X)$ and $\beta$ is a $n$-th root of $\alpha$ in $k(X)$. We are left with proving that we can extend $\beta$ to the closed point $p$. This can clearly be done, as if the ramification of $\beta$ at $p$ was nonzero, the same would be true for $\alpha$.
\end{proof}

Now all that is left is to compute the group of $\mu_n$-torsors of $\Brm G$.

\begin{lm}\label{H1triv}
Let $G$ be either ${\rm GL}_2$ or ${\rm PGL}_2 \times \Gm$. Then 
\[ {\rm H}^1_{\rm fppf}(\Brm G,\mu_n)=k^*/(k^*)^n.\]
\end{lm}
\begin{proof}
This is an easy consequence of the Kummer exact sequence and the fact that the Picard group of $\Brm G$ is torsion-free \cite{MV}.
\end{proof}

\begin{proof}[Proof of Theorem \ref{thm:Pic}]
We have seen that there is an open subset $\overline{\Hcal}_g '$ of $\overline{\Hcal}_g$ which is isomorphic to $\left[ U_g/G\right]$, where $U_g$ is an open subset of $\mathbb{A}^{2g+3}$ whose complement has dimension greater than $1$. By Lemmas \ref{H1inv}, \ref{H1triv} we know that $ {\rm H}^1_{\rm fppf}(\overline{\Hcal}_g ',\mu_n)=k^*/(k^*)^n$.

Moreover, $\overline{\Hcal}_g$ is a smooth, Deligne-Mumford stack \cite{Yam}*{Rmk. 1.4}, and it is closed inside $\overline{\mathcal{M}}_g$, so it is separated and has a presentation as a quotient by a linear algebraic group over $k$ coming from the well-known one of $\overline{\mathcal{M}}_g$ \cite{DelMum}. 

Using this, we can apply equivariant approximation and Lemma \ref{H1inv} to get smooth, separated schemes $X^0 \subset X$ such that $X^0$ is open in $X$ and that their first cohomology groups with coefficients in $\mu_n$ are isomorphic to those of $\overline{\Hcal}_g '$ and $\overline{\Hcal}_g$ respectively. Then Lemma \ref{H1inj} allows us to conclude immediately.
\end{proof}

\end{document}